%% file: Panazzolo-Silva-rev2.tex
\documentclass[11pt,psamsfonts]{amsart}

\usepackage[english]{babel}
\usepackage{amsmath,amssymb,graphicx,enumerate,bbm,caption}

\newcommand{\tmop}[1]{\ensuremath{\operatorname{#1}}}

\newenvironment{enumeratenumeric}{\begin{enumerate}[1.] }{\end{enumerate}}

\newcommand{\cR}{\ensuremath{\mathbbm{R}}}
\newcommand{\eps}{{\varepsilon}}

\usepackage{amsmath}
\usepackage{amsthm}
\usepackage{amssymb}
\usepackage{amscd}
\usepackage{amsfonts}
\usepackage{amsbsy}
\usepackage{epsfig,afterpage}
\usepackage{psfrag}
\usepackage{overpic}
\usepackage{latexsym}
\usepackage{pstcol}
\usepackage{graphicx}
\usepackage{mathtools}

\newcommand{\R}{\ensuremath{\mathbb{R}}}

\newcommand{\N}{\ensuremath{\mathbb{N}}}

\newcommand{\Z}{\ensuremath{\mathbb{Z}}}
\newcommand{\cS}{\mathbbm{S}}

\newcommand{\e}{\varepsilon}

\newcommand{\vF}{\mathcal{F}}
\newcommand{\vG}{\mathcal{G}}
\newcommand{\vZ}{\mathbf{Z}}
\newcommand{\vW}{\mathbf{W}}

\newcommand{\mini}{\mathrm{min}}

\newcommand{\NH}{\mathbf{NH}}
\newcommand{\Slide}{\mathrm{Slide}}
\newcommand{\Sew}{\mathrm{Sew}}
\newcommand{\reg}{\mathbf{ r}}
\newcommand{\sgn}{\operatorname{sgn}}
\newtheorem {theorem} {Theorem} [section]

\newtheorem {corollary}  {Corollary}[section]

\newtheorem {remark} {Remark}
\newenvironment{Remark} {\begin{remark} \rm }{\end{remark}}

\newtheorem {example} {Example}
\newenvironment{Example} {\begin{example}  \rm }{\end{example}}

\begin{document}

\title  [Regularization of Discontinuous Foliations]{Regularization of 
Discontinuous Foliations: Blowing up and Sliding Conditions via Fenichel Theory}

\author[   D. Panazzolo, P.R. da Silva]
{Daniel Panazzolo  $^{1,2}$ and Paulo R. da Silva $^3$ }

\address{$^1$  
Laboratoire de Math\'{e}matiques, Informatique et Applications--UHA,
4 Rue des Fr\`{e}res Lumi\`{e}re - 68093 Mulhouse, France}
\address{$^2$  
Universit\' {e} de Strasbourg, France}

\address{$^3$  Departamento de Matem\'{a}tica --
IBILCE--UNESP, Rua C. Colombo, 2265, CEP 15054--000 S. J. Rio Preto,
S\~ao Paulo, Brazil}

\email{daniel.panazzolo@uha.fr}
\email{prs@ibilce.unesp.br}

\date{}
\maketitle

\begin{abstract}
We study the regularization  of an oriented 1-foliation $\vF$ on $M \setminus \Sigma$ where $M$ is a  smooth manifold 
and $\Sigma \subset M$ is a closed subset, which can be interpreted as the discontinuity locus of $\vF$.
In the spirit of Filippov's work, we define a sliding and sewing dynamics on the discontinuity locus $\Sigma$ 
as some sort of limit of the dynamics of a nearby smooth 1-foliation and
obtain conditions  to identify whether a point belongs to the sliding or sewing regions.
\end{abstract}

\section{Introduction}\label{s0}

A 1-dimensional (singular) oriented foliation $\vF$ on a smooth manifold $M$ is defined by exhibiting an open covering 
of $M$ and a collection of smooth vector fields whose domains are the open sets of this covering, and which 
agree on the intersections of these open sets up to multiplication by a strictly positive function.  
A \emph{discontinuous 1-foliation} on  $M$ is given by  a closed subset 
$\Sigma \subset M$ with empty interior and a 1-dimensional oriented foliation 
on $M \setminus \Sigma$. \\

To fix the ideas  we start with the usual setting which was initially 
studied by Filippov  \cite{AF}. The  foliations considered 
are determined by flows of vector fields  expressed by
\begin{equation}\label{fili}
X = \Big(\frac{1+\sgn(f)}{2}\Big) X_+ + \Big(\frac{1-\sgn(f)}{2}\Big) X_-
\end{equation} 
for some smooth vector fields $X_+, X_-$ defined on $M$ and a function $f \in C^\infty(M)$ having $0$ as a regular value.  
The discontinuity locus is the smooth codimension one submanifold $\Sigma = f^{-1}(0)$.\\

In this setting, we say that a point $p\in\Sigma$ is \emph{$\Sigma$-regular} if  
$\mathcal{L}_{X_-}(f)\mathcal{L}_{X_+}(f)\neq 0$ 
and it is \emph{$\Sigma$-singular} if $\mathcal{L}_{X_-}(f)\mathcal{L}_{X_+}(f)= 0$.
Moreover the regular points are classified as \textit {sewing} if $\mathcal{L}_{X_-}(f)\mathcal{L}_{X_+}(f)>0$ 
or \textit{sliding} if $\mathcal{L}_{X_-}(f)\mathcal{L}_{X_+}(f)<0$.  \\

According Filippov's convention,  the flow of $X$ is easily determined in the neighborhood of  sewing points. 
Roughly speaking, it behaves like a constant vector field, as in\textit{ Flow Box Theorem}. However when 
a trajectory finds a sliding point,  the orbit remains in $\Sigma$ up to a $\Sigma$-singular point.
In the sliding region of $\Sigma$ the trajectory follows the flow determined by  a convex combination of $X^+$ and $X^-$, 
called  \textit{sliding vector field}.\\

\begin{figure}[!htb]
\epsfysize=4cm \centerline{\epsfbox{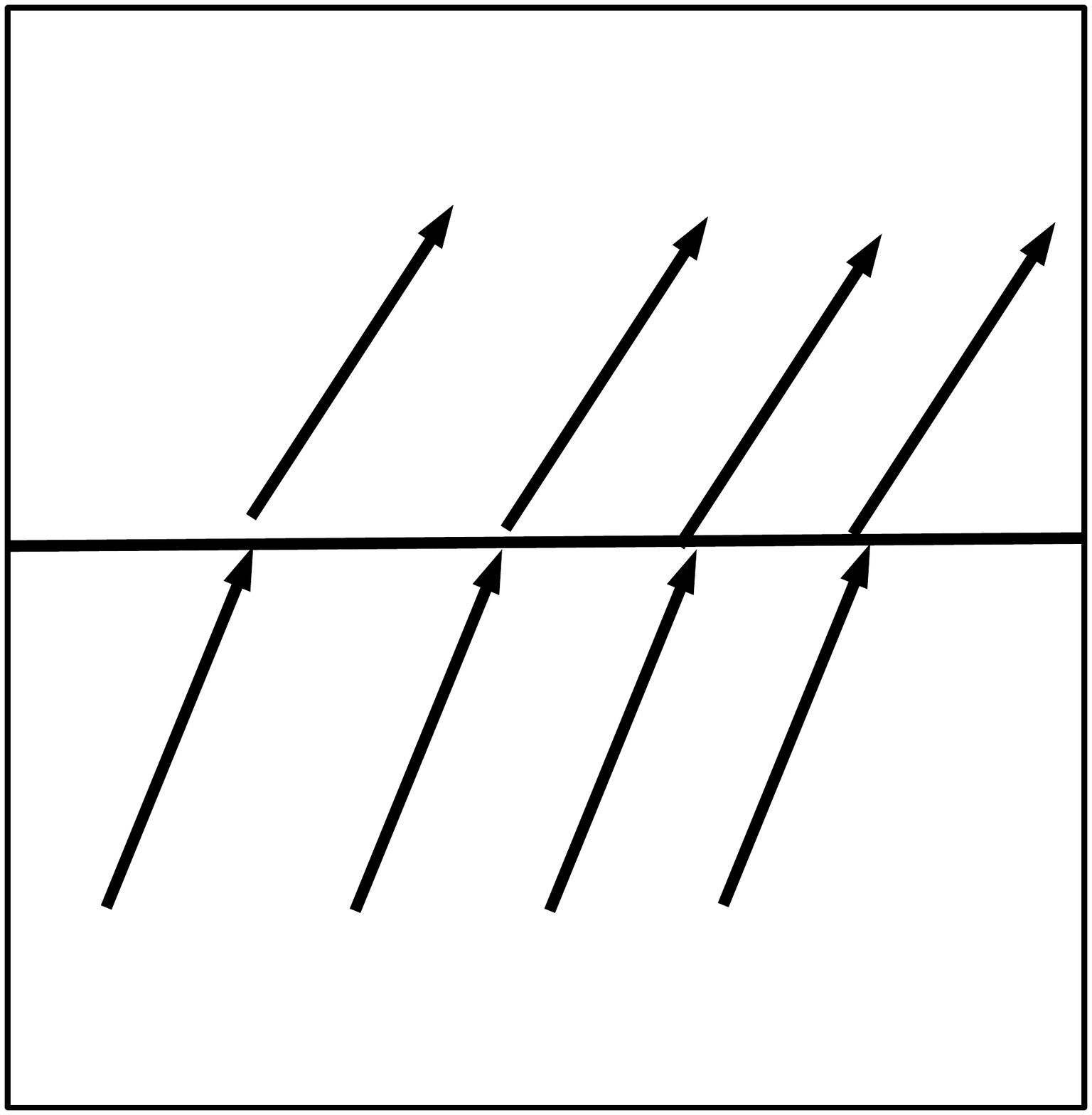}\quad\epsfbox{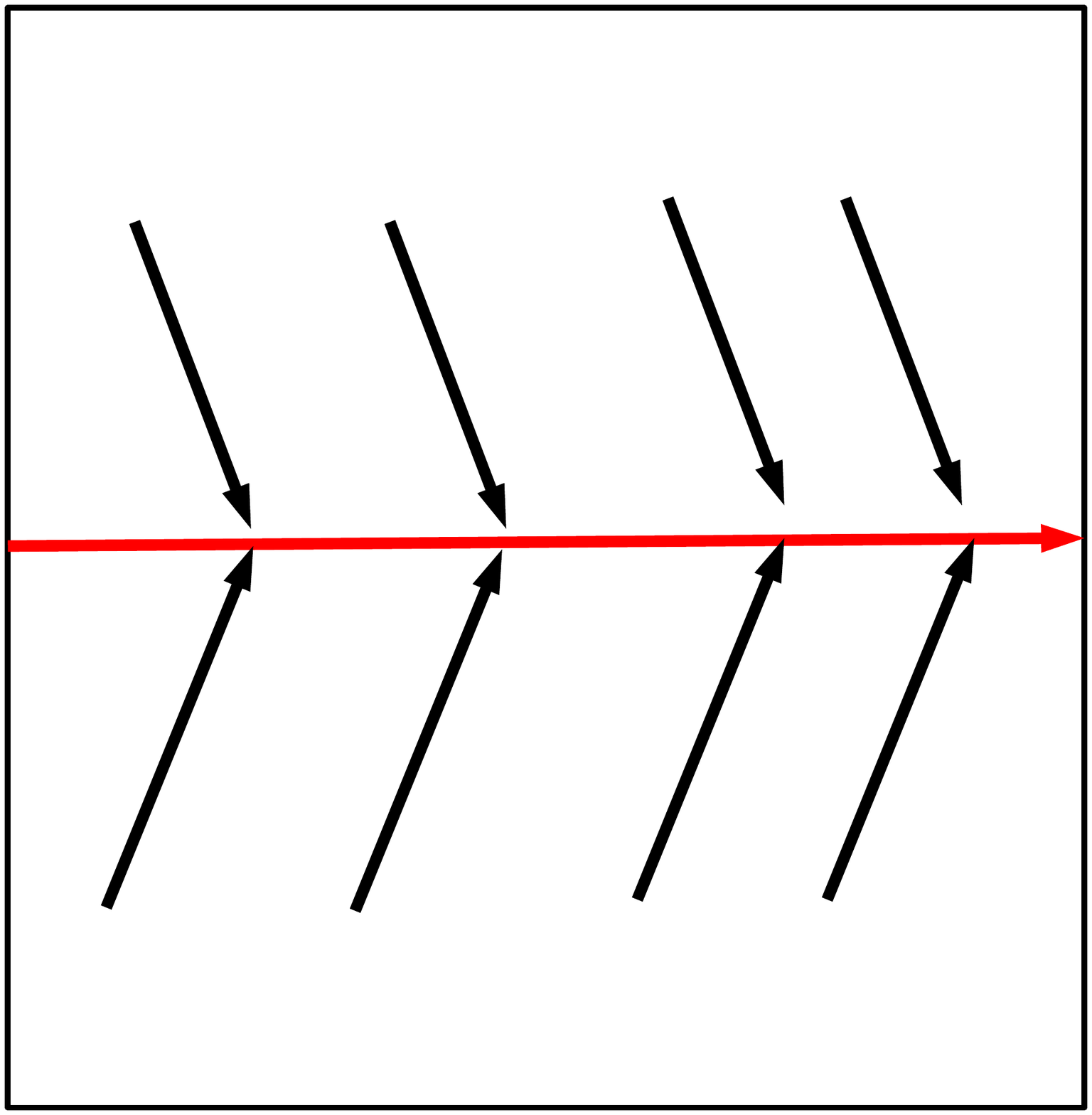}}
\caption{Sewing and sliding regions.
 }\label{rsew-rslid}
\end{figure}

These concepts do not have a natural generalization when the discontinuity occurs in singular sets, 
that is when $\Sigma = f^{-1}(0)$ is the inverse image of a critical  value. 
This is one of subjects which will be discussed in this article. \\

Our main tool in the study of  discontinuous foliation  is the \emph{regularization}. 
Basically, a regularization is a family of smooth vector fields $ X_{\e} $ depending on a parameter $\e > 0$ 
and such that $ X_{\e} $ converges uniformly to  $ X $ in each compact subset of $M \setminus \Sigma $ as $\e$ goes to zero.
One of the most well-known regularization process was introduced by Sotomayor and Teixeira \cite{LT,ST}.  
It is based on the use of   \emph{monotonic transition functions} $\varphi:\R\rightarrow\R$
(\footnote{by definition, this is a $C^\infty$ function such that $\varphi(t) = -1$ 
for $t \le -1$, $\varphi(t) = 1$ for $t \ge 1$ and $\varphi'(t) >0$ for $-1 < t < 1$.}).
The \emph{ST-regularization} of the vector field $X$ given in (\ref{fili}) is
the one parameter family 
\begin{equation} \label{STreg} X_{\e}= \frac{1}{2} \left(  1 + \varphi \left( \dfrac{f}{\e}\right)    \right)X_+ 
+ \frac{1}{2}\left ( 1 - \varphi \left( \dfrac{f}{\e}\right)   \right)   X_-. 
\end{equation}

The regularized vector field $X_{\e}$ is smooth for $\e > 0$ and 
satisfies that $X_{\e}=  X_+$ on $\{f>\e\}$ and $X_{\e}=  X_-$ on $\{f<-\e\}$. 
With this regularization process Sotomayor and Teixeira developed a systematic study of the 
singularities of these systems and 
also developed the Peixoto's program about structural stability. In particular, 
Teixeira analyzed the singularity of the kind
fold-fold, which was later known as $T$-singularity. We refer also \cite{Kr1,Kr2} for related problems.\\

In \cite{BPT}, the use of singular perturbation and blow-up techniques were introduced in the study 
of the ST-regularization.
Let us briefly describe this procedure, assuming for simplicity that $M = \R^2$ and that $\Sigma = \{y=0\}$. 
If we write $X_+=a_+\frac{\partial}{\partial x} + b_+ \frac{\partial}{\partial y}$ and 
$X_-=a_-\frac{\partial}{\partial x} + b_- \frac{\partial}{\partial y}$  then  
\[X_{\e}=\frac{1}{2}\left(a_++a_- +\varphi\left(  \dfrac{y}{\e}\right)\big( a_+-a_-\big)\right)\dfrac{\partial}{\partial x}+
\frac{1}{2} \left(b_++b_- +\varphi\left(\dfrac{y}{\e} \right)\big( b_+-b_-\big) \right)\dfrac{\partial}{\partial y}.\]
Considering the directional blow-up $y=\bar{\e}\bar{y}$, $\e=\bar{\e}$ we get the vector field
$$
\bar{X_{\e}}=\frac{1}{2}\left(a_++a_- +\varphi\left( \bar{y}\right)\big( a_+-a_-\big)\right)\dfrac{\partial}{\partial x}+
\frac{1}{2\bar{\e}} \left(b_++b_- +\varphi\left(\bar{y} \right)\big( b_+-b_-\big) \right)\dfrac{\partial}{\partial \bar{y}} 
$$
which corresponds to the singular perturbation problem  (\footnote{System \eqref{sp-problem} is called slow system and it is equivalent, up to a time reparametrization, to the fast system
 \[ 
 x' = \frac{\e}{2}\left(a_++a_- +\varphi\left( \bar{y}\right)\big( a_+-a_-\big)\right) \quad
 \bar{ y} =\frac{1}{2} \left(b_++b_- +\varphi\left(\bar{y} \right)\big( b_+-b_-\big) \right).
\]}).

\begin{equation} 
\label{sp-problem}
\left\{ \begin{array}{rcl}
 \dot x &=& \frac{1}{2}\left(a_++a_- +\varphi\left( \bar{y}\right)\big( a_+-a_-\big)\right) \vspace{0.5cm}\\
\bar{\e}\,\dot {\bar y} &=& \frac{1}{2} \left(b_++b_- +\varphi\left(\bar{y} \right)\big( b_+-b_-\big) \right)
\end{array}
\right.
\end{equation}

For  $\bar{\e}=0$, the slow manifold of \eqref{sp-problem} is the set implicitly defined by 
\[(b_++b_-)+\varphi\left(\bar{y}\right)(b_++b_-)=0\]
with slow flow determined by 
\[\frac{1}{2}\left(a_++a_- +\varphi\left( \bar{y}\right)\big( a_+-a_-\big) \right)\dfrac{\partial}{\partial x}.\]  
Silva et all  \cite{LST2} proved 
that the set of sliding points, according  Filippov convention, is the projection
of the slow manifold of \eqref{sp-problem} on $\Sigma$.  Moreover 
they proved that the slow flow of \eqref{sp-problem} and the sliding vector field idealized by  Filippov have the same equation.\\

For better visualization, we use the polar blow up $ y= r \cos \theta, \e = r \sin  \theta $ 
with $ \theta \in (0, \pi) $. 
In this case the discontinuity $ \Sigma $ is replaced by a semi-cylinder on which we 
draw the slow manifold and the fast and slow 
trajectories. See figure \ref{buNew}.\\

The singular perturbation problem which is obtained evidently depends on the choice of the regularization. 
The sliding vector field idealized  by 
Filippov appears when we consider the ST-regularization, see for instance \cite{LST,LST2,LST3,LST4}. 
However Novaes and his collaborators \cite{NJ} have 
considered   a slightly more general regularization, called   non-linear regularization, 
which produces singular perturbation problem with slow manifold
having fold points and thus not  defining only one possible sliding flow.
It seems evident that other regularizations may produce new sliding regions.\\

The techniques of singular perturbation have also been applied to deal with 
discontinuities on surfaces with singularities. 
Teixeira and his collaborators realized that in the case where $\Sigma$ has 
a transverse self-intersection 
a process of double regularization can be used, and that it generates systems with multiple time scales.\\

In this work we intend to unify the different approaches of the previous works. 
Let us briefly summarize the results proved in this paper.\\

Given a  smooth manifold $ M $, initially we  introduce the concepts of 
\emph {1}-dimensional oriented foliation on $ M $ and  \emph {discontinuous 1-foliation} on 
$ M $ with 
discontinuity locus $ \Sigma $. The first question we address is to get conditions so 
that the  foliation can be {\em smoothed }by a sequence of blowing-ups.\\

Our first results are the  following:\\

\begin{itemize}
\item If $\vF$ is a piecewise smooth foliation and the discontinuity locus $\Sigma$ is a smooth 
submanifold of codimension one then the foliation $\vF$ is blow-up smoothable.
See \textit{Theorem \ref{theorem-smoothlocus}}.

\item If the  discontinuity locus $\Sigma$ is a globally defined analytic subset then
there is a piecewise smooth 1-foliation  which is related to the initial foliation by a sequence of blow-ups  and 
whose discontinuity locus is smooth.  Moreover if we further suppose  
that the discontinuity locus  has codimension one then the foliation is blow-up smoothable.  
See \textit{Theorem \ref{theorem-singularlocus} }and \textit{Corolary \ref{corr}}.
\end{itemize}

\begin{figure}[!htb] 
\epsfysize=8cm \centerline{\epsfbox{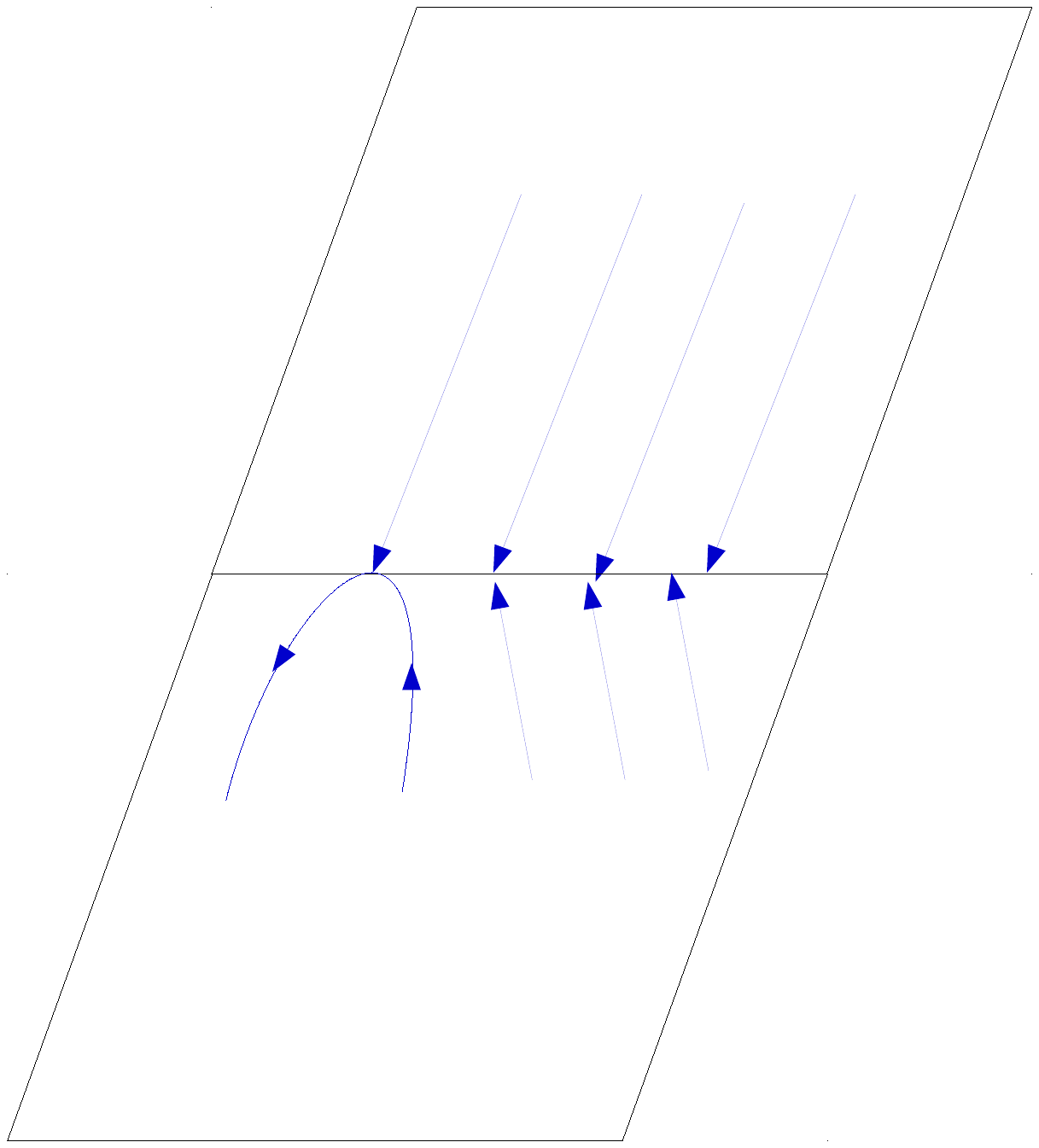} \epsfbox{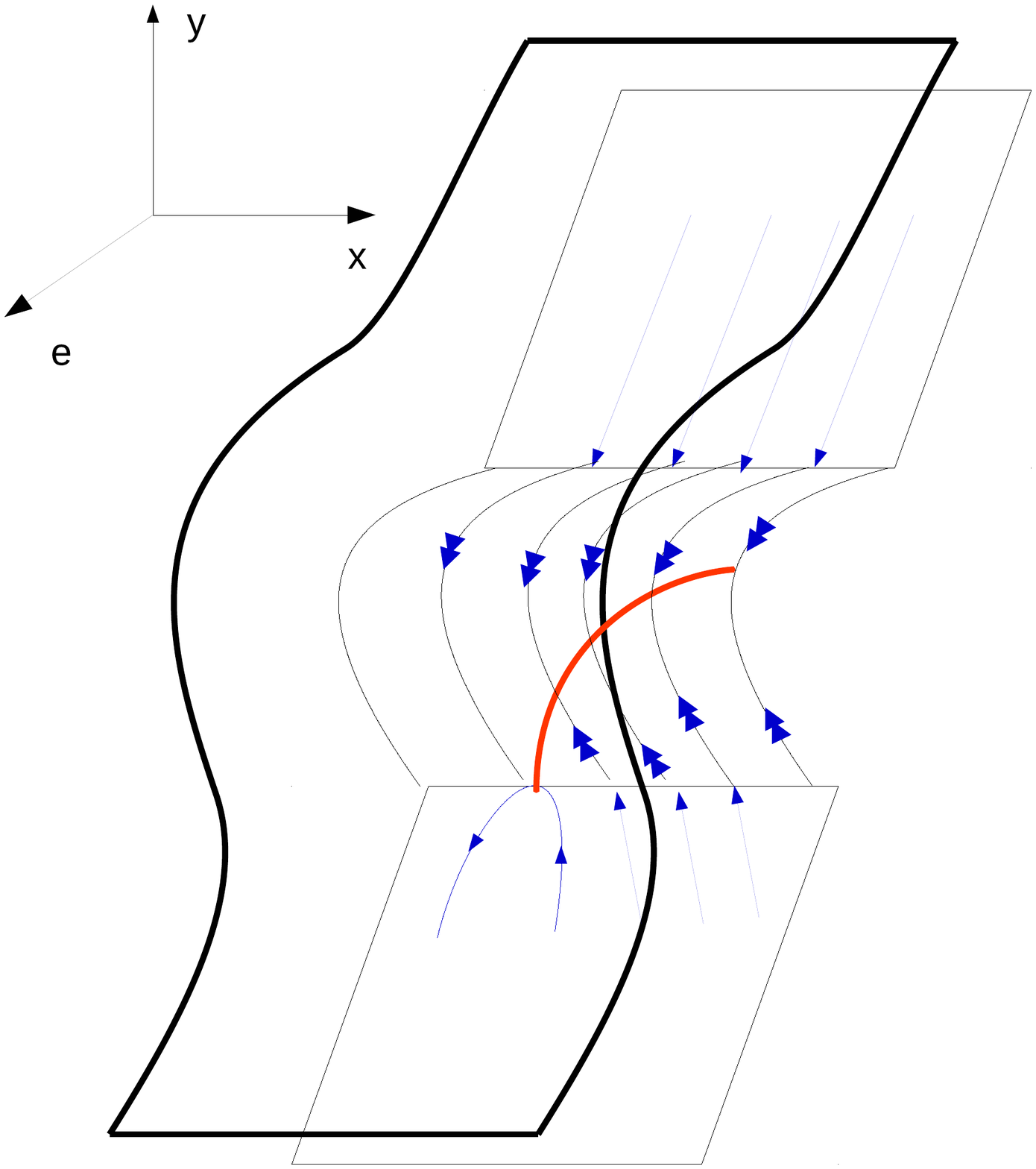}} \caption{ Blow-up smoothing  a regularization of transition type.}
\label{buNew}
\end{figure}

The smoothing procedure defined in Section \ref{s1} does not allow to define the so-called {\em sliding dynamics} 
along the discontinuity locus.  
So, we define such sliding dynamics as the limit of the 
dynamics of a nearby smooth 1-foliations. This leads us to introduce a general notion of regularization for piecewise-smooth 1-foliations. \\

Roughly speaking, the regularization is given by a new foliation 
depending on a parameter $\eps$, which is smooth for $\eps > 0$ and which coincides with 
the original discontinuous 1-foliation when $\eps$ equals zero.  We generalize 
the notion of ST-regularization to this global context and consider a larger family of 
regularizations (called of {\em transition type}) by dropping the condition of  monotonicity of the transition function. 

Basically, the sliding region associated to a given regularization is defined as the accumulation set of invariant manifolds
of the regularized system.  We prove the following results.

\begin{itemize}
\item The regularization of transition type is  blow-up smoothable. See \textit{Theorem \ref{theorem-smoothableST}.}
\item We obtain conditions on the transition function to identify whether a point lies in the sliding region.
See \textit{Theorems \ref{theorem-slide}} and \textit{\ref{theorem-sewingreg}}.\\
\end{itemize}

 The paper is organized as follows. In Section \ref{s1} we give the preliminary definitions  and prove Theorems \ref{theorem-smoothlocus},
  \ref{theorem-singularlocus} and Corolary \ref{corr}.
In Section \ref{s2} we study the regularization and in Section \ref{s3} we combine the blowing-up technique and 
the Fenichel's theory to give sufficient conditions for identifying the sliding region.\\

Figure \ref{figCORNER} is a pictorial representation of the blow-up smoothing process. 
It shows a piecewise smooth discontinuous foliation with analytic discontinuous set 
having a smooth part and a singular one.  Applying a  regularization of the kind transition  we get a new foliation $\mathcal{F}^r$.
In the figure we draw the level $\mathcal{F}^0$ of this foliation. The leaves displayed in $\mathcal{F}^0$ 
are the trajectories of the singular perturbation problem \eqref {sp-problem}. The simple arrows correspond to the slow flow 
and the double arrows correspond to the fast flow, which is obtained after a time reparametrization.

\begin{figure}[!htb]
\epsfysize=8cm \centerline{\epsfbox{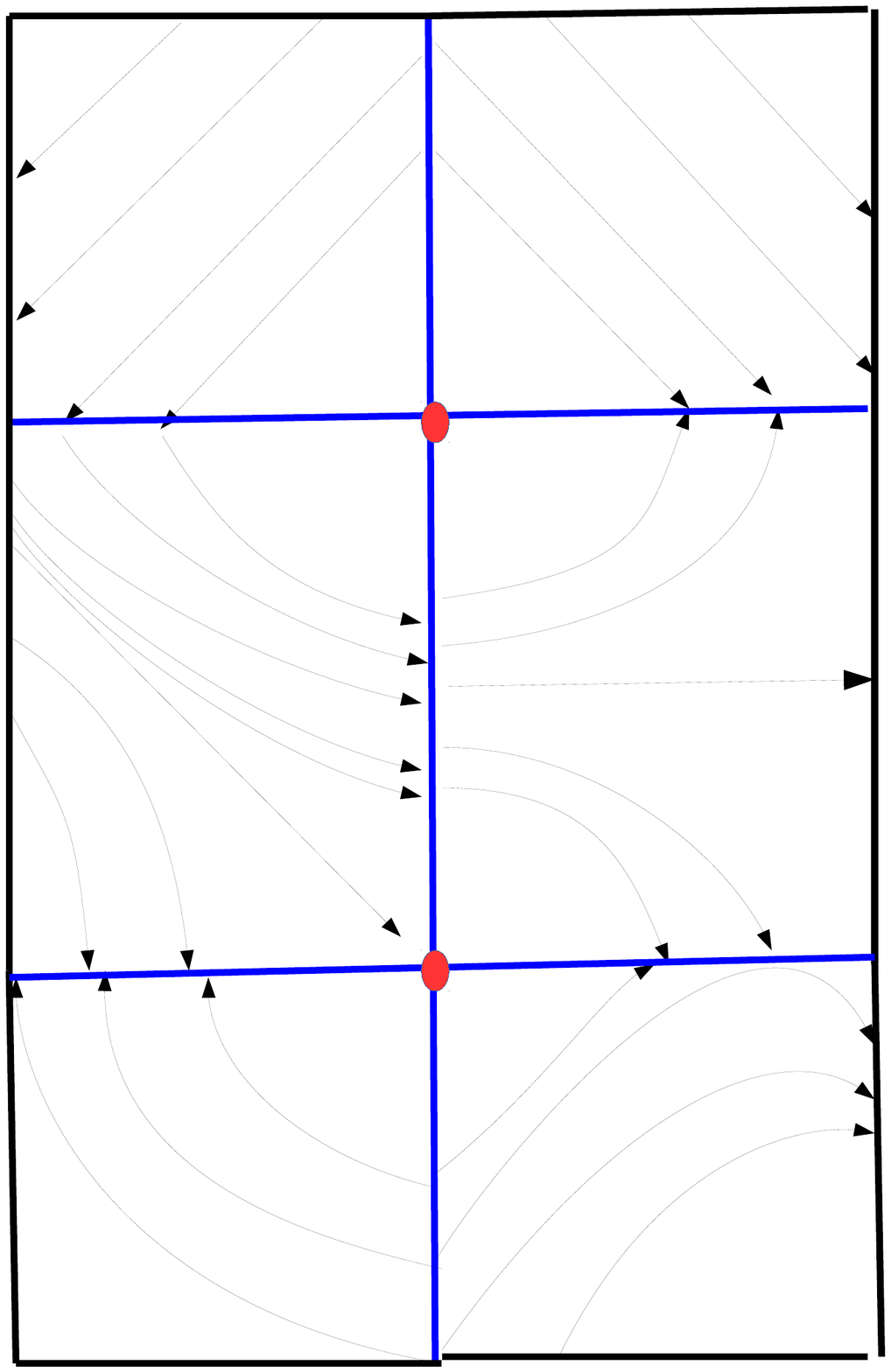}\epsfbox{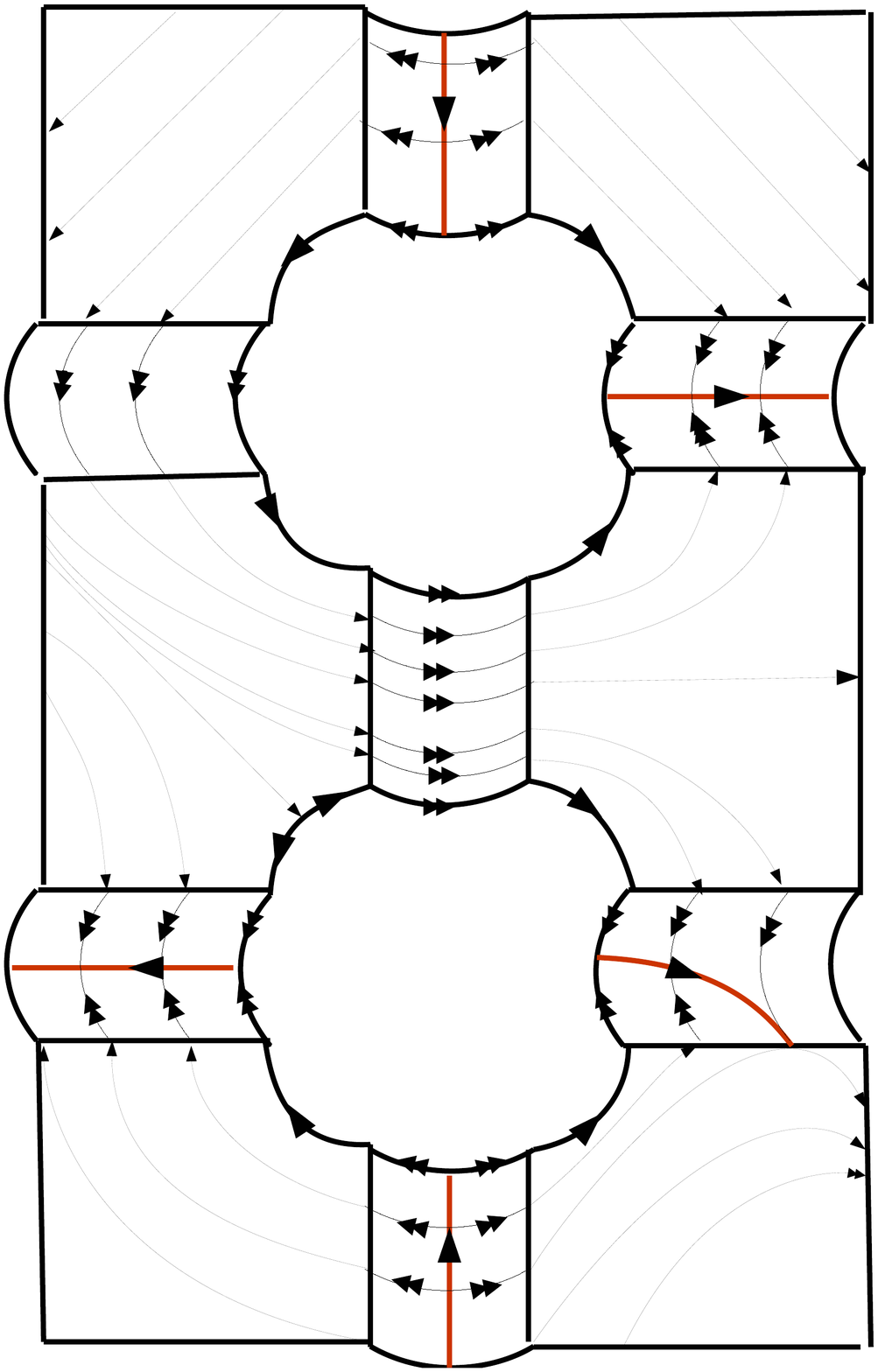}}
\caption{On left we have a discontinuous 1-foliation $\mathcal{F}$, the blue line is $\Sigma^{smooth}$ and the red point
is $\Sigma^{sing}$. On right we see the  leaf   $\mathcal{F}^0$ of the  regularization $\mathcal{F}^r$. The red line
is the slow manifold of the corresponding singular perturbation problem. 
Simple arrow means slow flow  and double arrow the fast one.}
\label{figCORNER}
\end{figure}

\section{Discontinuous 1-foliations}\label{s1}
 In this section we present the preliminary definitions related to discontinuous 1-foliations. Also we enunciate and prove results 
 related to blow-up smoothing of discontinuous 1-foliations.
  
\subsection{Smoothable discontinous foliations}
We work in the category of manifolds with corners.  We briefly recall that a manifold with corners of dimension 
$n$ is a paracompact Hausdorff space with a smooth structure which is locally modeled by open subsets of
$
(\R^+)^k \times \R^{n-k}
$. We refer the reader to \cite{Joy,M} for a careful exposition.

Let $M$ be a smooth manifold (with corners).  A smooth vector field defined on $M$ will 
be called {\em non-flat }if its Taylor expansion is non-vanishing at each point of $M$.  
From now on, all vector fields we will consider are non-flat.

We say that a pair $(V,Y)$ formed by an open set $V \subset M$ and a smooth vector field $Y$ 
defined in $V$ is a \emph{local vector field }in $M$.

A \emph{1-dimensional oriented foliation }on $M$ is a collection
$$
\vF = \{(U_i,X_i)\}_{i \in I}
$$
of local vector fields such that:
\begin{enumeratenumeric}
\item  $\{U_i\}$ is an open covering of $M$. 
\item For each pair $i,j \in I$, 
\begin{equation}\label{transition-function}
X_i = \varphi_{ij} X_j
\end{equation} 
for some strictly positive smooth function $\varphi_{ij}$ defined on $U_i\cap U_j$. 
\end{enumeratenumeric}
\begin{Remark}
The importance to consider {\em  oriented foliations }instead of globally defined vector fields in the manifold $M$ 
will become clear later.  Roughly speaking, even if our initial object is a foliation globally defined by a smooth vector 
field, this property will not necessarily hold after the blowing-up operation.
\end{Remark}
We say that a local vector field $(V,Y)$ is a \emph{local generator }of the foliation $\vF$ if the augmented collection
$$
\{(U_i,X_i)\}_{i \in I} \; \cup \; \{ (V,Y)\}
$$ 
also satisfies conditions 1. and 2. of the above definition.  From now on, we will 
suppose that the collection $\vF$ is {\em saturared}, meaning
that it contains all such local generators.

Let $\psi : N \rightarrow M$ be a smooth diffeomorphism between two manifolds $N$ and $M$.  We will say 
that two 1-dimensional oriented foliations $\vF$ and $\vG$ defined respectively in $M$ and $N$ 
are \emph{Êrelated by $\psi$} if for each local vector field $(V,Y)$ which is a generator 
of $\vG$, the push-forward of this local vector field under $\psi$, namely 
$$
(U,X) := \big( \psi(V), \psi_* Y \big),
$$
is a generator of $\vF$.

A  \emph{possibly discontinuous 1-foliation on a manifold $M$} is given by  a closed subset $\Sigma \subset M$ 
with empty interior and a 1-dimensional oriented foliation $\vF$ defined in $M \setminus \Sigma$.  

The set $\Sigma$ is called the {\em discontinuity locus }of $\vF$. We can write the decomposition 
$$\Sigma = \Sigma^{\tmop{smooth}} \cup \Sigma^{\tmop{sing}}$$
where $\Sigma^{\tmop{smooth}}$ denotes the subset of points where $\Sigma$
locally coincides with an embedded submanifold of $M$. 
We shall say that $\vF$ has a {\em smooth discontinuity
  locus }if $\Sigma = \Sigma^{\tmop{smoth}}$.
  
\begin{Example}\label{example-1}
The vector field in $\R^2$ given by 
\begin{equation}\label{sign-vf}
X = \frac{\partial}{\partial x} - \sgn(y) \frac{\partial}{\partial y}
\end{equation}
defines a discontinuous 1-foliation which has discontinuity locus $\Sigma = \{y = 0\}$.  
\end{Example}

\begin{Example}\label{example-2}
Consider the discontinuous 1-foliation defined in $\R^2$ by the vector field
\begin{equation}\label{radial-vf}
X = \frac{x^2-y^2}{x^2+y^2} \left( -y \frac{\partial}{\partial x} + x \frac{\partial}{\partial y}\right) +
 \left( x \frac{\partial}{\partial x} + y \frac{\partial}{\partial y}\right)
\end{equation}
with discontinuity locus $\Sigma = \{0\}$.  Notice that $X$ is not smooth at the origin.
\end{Example}

\begin{Example}\label{example-3}
Consider the discontinuous 1-foliation defined in $\R^3$ by the vector field
\begin{equation}\label{sign2-vf}
X = - \sgn(x)  \frac{\partial}{\partial x} - \sgn(y) \frac{\partial}{\partial y} + \frac{\partial}{\partial z}
\end{equation}
with has the (non-smooth) discontinuity locus $\Sigma = \{xy=0\}$. 

\begin{figure}[htbp]
\begin{center}
{ 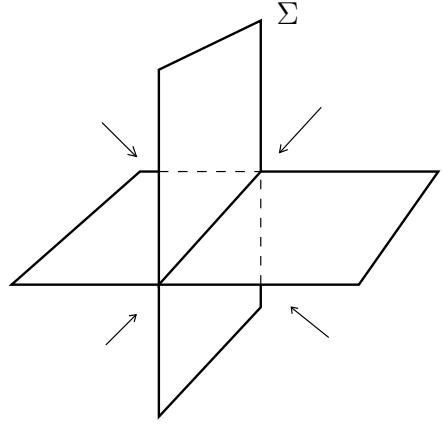}
\caption {Discontinuity locus of the 1-foliation of  the example \ref{example-3}.}
\label{example3-fig}
\end{center}
\end{figure}

\end{Example}
More generally, let $f$ be an arbitrary smooth function on a manifold $M$ and let $X_+$, $X_-$ 
be two smooth vector fields defined on $M$.  Then
\begin{equation}\label{equation-generaldiscont}
X = \Big(\frac{1+\sgn(f)}{2}\Big) X_+ + \Big(\frac{1-\sgn(f)}{2}\Big) X_-
\end{equation}
is a discontinuous 1-foliation with discontinuity locus $\Sigma = f^{-1}(0)$. 
\begin{Example}\label{example-generaldiscont}
The figure \ref{example4-fig} illustrates 
a discontinuous 1-foliation in $M = \mathbb{R}^2$, where $f(x,y) = y^2 - x^2(x + 1)$,  $X_+ = -y \frac{\partial}{\partial x} + x \frac{\partial}{\partial y}$ 
and $X_- = Ê-X_+$.
\begin{figure}[htbp]
\begin{center}
{ 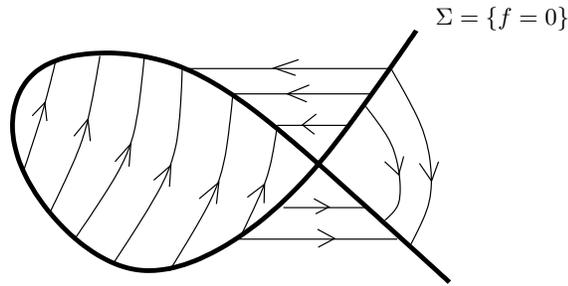}
\caption{Discontinuous 1-foliation of  the example \ref{example-generaldiscont}.}
\label{example4-fig}
\end{center}
\end{figure}
\end{Example}

As the above examples show, there is no reason to expect that $\vF$ can be extended smoothly 
(or even continuously) to $\Sigma$.  In order to circumvent this difficulty, one possibility is to modify 
the ambient space $M$ by a {\em blowing-up }(or a sequence of blowing-ups) and to expect that the
 modified foliation extends smoothly to the whole ambient space.  We refer the reader to \cite{M}, 
 chapter 5 for a detailed definition of the blowing-up operation in the category of manifold with corners.

We will say that discontinuous 1-foliation $\vF$ defined in $M$ and with discontinuity locus
 $\Sigma$ is {\em blow-up smoothable }if there exists a locally finite sequence of blowing-ups (with smooth centers)
$$
M = M_0 \longleftarrow \cdots \longleftarrow \cdots \longleftarrow  \widetilde{M}
$$
and a smooth 1-foliation $\widetilde{\vF}$ defined in $\widetilde{M}$ such that:
\begin{enumeratenumeric}
\item  The map $\Phi :  \widetilde{M} \rightarrow M$ is a diffeomorphism outside $\Phi^{-1}(\Sigma)$, and
\item $\widetilde{\vF}$ and $\vF$ are related by $\Phi$, seen as a map from 
$\widetilde{M} \setminus \Phi^{-1}(\Sigma)$ to $M \setminus \Sigma$.
\end{enumeratenumeric}
Let us show that the examples studied above are blow-up smoothable.

\begin{Example}
Consider the discontinuous 1-foliation defined in $\R^2$ by the vector field in (\ref{sign-vf}), 
which has discontinuity locus $\Sigma = \{y = 0\}$.  If we denote by $\cS^0 = \{\pm 1\}$ the $0^{th}$-sphere, 
the blowing-up of $\Sigma$ is defined by the map
$$
\begin{array}{cccl}
\Phi\; : & \R \times (\cS^0\times\R_{\ge 0}) & \longrightarrow & \R^2\\
& (u,\;\;(\pm1, r)) & \longmapsto & x = u,\, y = \pm r
\end{array}.
$$ 
The resulting 1-foliation in $\R \times (\cS^0\times\R_{\ge 0})$, given by 
$\frac{\partial}{\partial u} \mp \frac{\partial}{\partial r}$  is clearly smooth.
\end{Example}

\begin{Example}
Consider the discontinuous foliation defined by (\ref{radial-vf}). The blowing up of
 the origin is defined by the polar coordinates map
$$
\begin{array}{cccl}
\Phi\; : \;& \cS^1\times\R_{\ge 0} & \longrightarrow & \R^2\\
& (\theta,r) & \longmapsto & x = r \cos(\theta),\,  y = r \sin(\theta)
\end{array}
$$ 
and an easy computation shows that this foliations is mapped to 
$$
\cos(2\theta) \frac{\partial}{\partial \theta} + r \frac{\partial}{\partial r}
$$
which is clearly a $C^\infty$ vector field on $\cS^1\times\R_{\ge 0}$.
\end{Example}

\begin{Example}
Consider the discontinuous foliation defined by (\ref{sign2-vf}). This foliation is blow-up 
smoothable by a sequence of three blowing-ups:  
$$
\R^3 = M_0 \stackrel{\Phi_1}{\longleftarrow} M_1\stackrel{\Phi_1}{\longleftarrow} M_2\stackrel{\Phi_1}{\longleftarrow} M_3
$$
with respective centers given by the $z$-axis,  the strict transform of the hyperplane 
$\{x = 0\}$ and the strict transform of hyperplane $\{ y = 0\}$.
\end{Example}
Notice however that there are discontinuous 1-foliations which are not blow-up smoothable.
\begin{Example}
Consider the discontinuous 1-foliation in $\mathbb{R}^2$ defined as in (\ref{equation-generaldiscont}), where we take
$$
f(x,y) = y^2 - e^{-\frac{1}{x^2}} \sin^2\big(1/x)
$$
and $X_+ = \partial/\partial y$, $X_- = \partial/\partial x$.  The set $\mathbb{R}^2 \setminus f^{-1}(0)$ 
has an infinite number of open connected components in any neighborhood of the origin, and this property cannot be destroyed by a locally finite 
sequence of blowing-ups.
\begin{figure}[htbp]
\begin{center}
{ 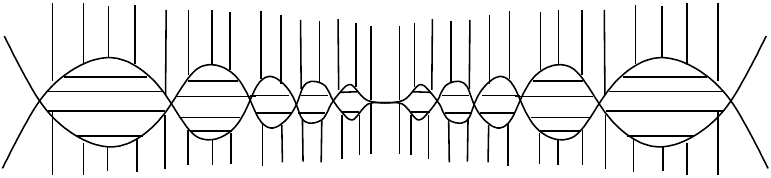}
\caption{Discontinuous 1-foliation which is not blow-up smoothable.}
\label{examplenonsmooth-fig}
\end{center}
\end{figure}

\end{Example}

\subsection{Piecewise smooth 1-foliations}\label{subsect-piecewise-smooth}
A natural problem which arises is to establish conditions which guarantee that a discontinuous 
1-foliation is blow-up smoothable.  For this, we will introduce a particular class of discontinuous 
1-foliations where, roughly speaking, we require that firstly each local vector field which defines
 such foliations extends smoothly to the discontinuity locus, and secondly that the discontinuity locus 
 is an analytic subset of $M$.

More formally, let $\vF$ be a discontinuous 1-foliation defined on a manifold $M$ and with 
discontinuity locus $\Sigma$.  A {\em local multi-generator }of $\vF$ is a pair $(U,\{X_1,\ldots,X_k\})$ 
satisfying the following conditions: 

\begin{enumeratenumeric}
\item $U$ is an open set of $M$ and we can write $U \setminus \Sigma$ as a finite disjoint union 
\begin{equation}\label{componentsofSigma}
U \setminus \Sigma = U_1 \sqcup \cdots \sqcup U_k
\end{equation}
of open sets $U_1,..,U_k$.
\item For each $i = 1,..,k$, $X_i$ is a smooth vector field {\em defined in $U$} and such that 
\begin{equation}\label{componentsofSigma-generators}
(U_i, X_i|_{U_i}) \text{Ê is a local generator of $\vF$.}
\end{equation}
\end{enumeratenumeric}

We will say that  $\vF$ is {\em piecewise smooth }if there exists a collection $\mathcal{C}$ of 
local multi-generators as above whose domain forms an open covering of  $\Sigma$ and the 
following  compatibility condition holds: For each two local multi-generators
$$(U,\{X_1,\ldots,X_k\}),\quad (V,\{Y_1,\ldots,Y_l\})$$
belonging to $\mathcal{C}$, there exists a strictly positive
smooth function $\varphi$ defined in $U \cap V$ such that
\begin{equation}\label{local-compatibilityofgen}
X_i = \varphi\, Y_j \quad \text{ on }\quad U_i \cap V_j
\end{equation}
for each pair of indices $i= 1,\ldots,k$ and $j = 1,\ldots,l$.

\begin{Remark}
In what follows, it will be important to require the transition function $\varphi$ to 
be {\em the same }on all intersections $U_i \cap V_j$.   
\end{Remark}
\begin{figure}[htbp]
\begin{center}
{ 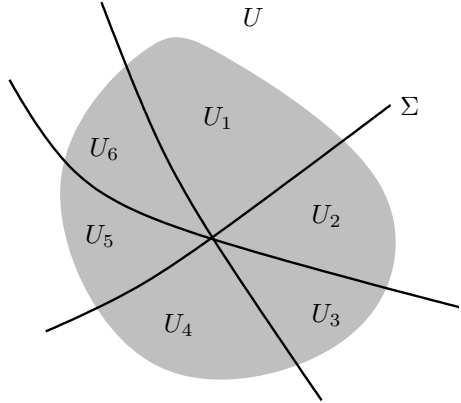}
\label{multi-generator}
\caption{The neighborhood of a point in $\Sigma$.}
\end{center}
\end{figure}

\begin{Example}The Example \ref{example-1} exhibits a piecewise smooth 1-foliation, since $\vF$ 
is defined simply by restricting the constant vector fields
$$
X_+ = \frac{\partial}{\partial x} - \frac{\partial}{\partial y} \qquad \text{and}\qquad X_- = \frac{\partial}{\partial x} + \frac{\partial}{\partial y}
$$
to the subsets $U_+ = \{y > 0\}$ and $U_- = \{y < 0\}$, respectively.  Similarly, the 
1-foliation in Example~\ref{example-3}Ê
 is piecewise smooth.  On the other hand, the discontinuous vector field $X$ in 
 Example \ref{example-2} is not
  piecewise smooth. In fact, the vector field $X$ is not smooth at the origin and therefore 
  the foliation  can not be smoothly extended to $ \Sigma = \{(0,0) \}.$
\end{Example}

A simple consequence of the definition of piecewise smooth 1-foliations is the following:

\begin{theorem}\label{theorem-smoothlocus}
Let $\vF$ be a piecewise smooth 1-foliation on a manifold $M$ whose discontinuity locus $\Sigma$ is an
smooth submanifold of codimension one. Then, $\vF$ is blow-up smoothable.
\end{theorem}

\begin{proof}
We consider the smooth map $\Phi : N \rightarrow M$ defined by the blowing-up with center 
$\Sigma$, and exceptional divisor $D = \Phi^{-1}(\Sigma)$.
The smooth foliation $\vG$ in $N$ is now defined by describing its local generator 
at each point $q\in N$.  
We consider separately the case where $q \in N\setminus D$ and $q \in D$.  
In the former case, we can
 choose a generator $(U,X)$ of $\vF$ defined in a sufficiently small neighborhood 
 an of $p = \Phi(q)$ and decree that
$$
(V,Y) = \big( \Phi^{-1}(V),\Phi^*Y \big)
$$
is a local generator of $\vG$ near $q$.  Notice that this construction defines 
unambiguously the 1-foliation 
on $N \setminus D$, since it is independent of the choice of $(U,X)$.

Suppose now that $q \in D$, i.e. that $p = \Phi(q)$ lies in $\Sigma$.  
Then, we can choose local coordinates 
$(x_1,..,x_{n-1},y)$ in a neighborhood $U$ of $p$ such that $\Sigma = \{y = 0\}$ 
and the blowing-up 
map assumes the form
$$
\begin{array}{rcl}
(\cS^0\times \R^+) \times \R^{n-1} & \longrightarrow & \R^n\\
 (\pm 1,r), (u_1,...,u_{n-1}) & \longmapsto & y = \pm r,\, x_1 = u_1,\,\ldots,\, x_{n-1} = u_{n-1}.
\end{array}
$$ 
Up to reducing $U$ to some smaller neighborhood of $p$, we can assume that $U \setminus \Sigma$ 
can be written as the disjoint union of connected subsets
$U_+ = \{y > 0\}$, $U_- =\{ y < 0\}$ and that there exists two smooth vector fields $X_+,X_-$ defined on $U$ 
such that $(U_+,X_+)$ and $(U_-,X_-)$ are local generators of $\vF$.  
Now, by the expression of the blowing-up
 map either $V_+ = \Phi^{-1}(U_+)$ or $V_- = \Phi^{-1}(U_-)$ is an open 
 neighborhood of $q$ in $N$.
Therefore, according to the choice of the $\pm$ sign, we decree that the 
local vector field $(V_\pm,\Phi^* X_\pm)$ 
is a local generator of $\vG$ at $q$.

This procedure defines in an unambiguous way a smooth 1-foliation $\vG$ on $N$, 
which is moreover related to $\vF$ by $\Phi$. 
\end{proof}
A natural question which arises is whether the above result can be generalized to the case where $\Sigma$ is a smooth submanifold of higher codimension.
\begin{Example}\label{example-flatcone}
Consider the piecewise-smooth 1-foliation on $\mathbb{R}^3$  defined by
$$X = z^k\frac{\partial}{\partial x} + \varphi\left(\frac{y z^l}{e^{-1/z^2}}\right) \frac{\partial}{\partial y} \quad \text{ if $z \ne 0$}, \quad X = z^k \frac{\partial}{\partial x} + \sgn(y) \frac{\partial}{\partial y},\quad \text{ if $z = 0$}.$$
where $\varphi$ is a monotonic transition function as defined in the Introduction and $k,l \in \mathbb{N}$ are arbitrary positive integers.  Notice that 
$X$ is smooth outside the discontinuity locus $\Sigma = \{y=z=0\}$, which is of codimension 2.  A blowing up with center the origin will produce (in the $z$-directional chart) the same 
expression with the integer $l$ replaced by $l+1$.   Similarly, a blowing-up with center $\Sigma$ will produce 
exactly the same expression with $k$ and $l$ replaced by $k+1$ and $l+1$ respectively.  
\begin{figure}[htbp]
\begin{center}
{ 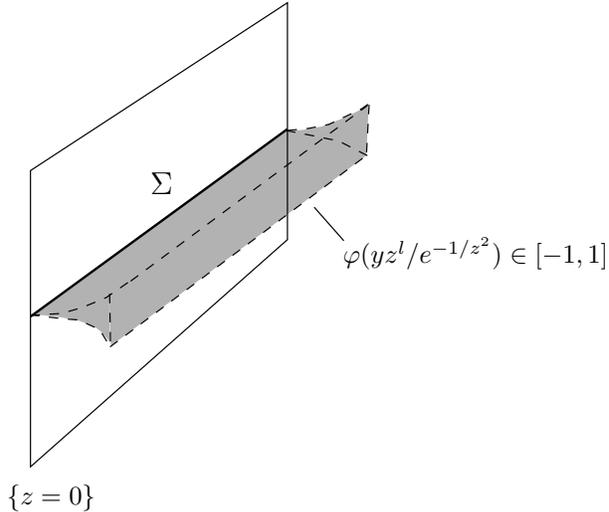}
\label{flatcone-fig}
\caption{In the Example \ref{example-flatcone}, the region $\{(x,y,z) : -1\le \varphi\left(\frac{y z^l}{e^{-1/z^2}}\right)\le 1\}$ is a flat cone with edge on $\Sigma$.}
\end{center}
\end{figure}
Therefore, 
no sequence of blowing-ups will allow a $C^\infty$ extension of this vector field to the exceptional divisor.
\end{Example}

Our next goal is to obtain a similar result in the case where $\Sigma$ is a codimension one {\em singular subvariety}.  
For this, we need to impose another regularity condition, which will allow us to use 
the Theorem of Resolution of Singularities.

We will say that $\vF$ has an {\em analytic discontinuity locus }if the ambient space 
$M$ is an analytic 
manifold and the discontinuity locus $\Sigma$ is a {\em globally defined analytic subset }of $M$.  
In other words, we 
assume $\Sigma = \Sigma(f)$ is the vanishing locus of a finite collection of global analytic functions 
$f = (f_1,\ldots,f_m)$ defined on $M$ (\footnote{According to \cite{C}, Proposition 15, and Grauert's 
embedding theorem \cite{Gr}, this is equivalent to say $\Sigma$ is the vanishing locus of a coherent 
sheaf of ideals $\mathcal{I}$ defined on $M$.}). 

Under the above hypothesis, there exists an unique filtration by semianalytic sets (see \cite{Lo})
\begin{equation}\label{stratification}
\Sigma^0  \subset \Sigma^1 \subset \cdots \subset \Sigma^d = \Sigma
\end{equation}
where, for each $k = 1,\ldots,d$, the set $\Sigma^k \setminus \Sigma^{k-1}$ is a smooth manifold of dimension $k$.  
Using this decomposition, we say that $d$ is the {\em dimension }of $\Sigma$ and that $\Sigma^{\tmop{reg}} = \Sigma^d \setminus \Sigma^{d-1}$
 is the {\em regular part }of $\Sigma$.   The complementary set $\Sigma^{\tmop{exc}}  = \Sigma \setminus \Sigma^{\tmop{reg}}$ is 
 called the {\em exceptional locus}.
\begin{Remark}
We observe that, in general, the inclusion $\Sigma^{\tmop{reg}}\subset\Sigma^{\tmop{smooth}}$ is strict.  For instance, 
the Whitney umbrella $\Sigma= \{z^2 - x y^2 = 0\}$ is such that $\Sigma^{\tmop{exc}} = \{y=z=0\}$ contains strictly 
$\Sigma^{\tmop{sing}} = \{x \ge 0, y=z=0\}$.  Taking the complementaries it follows that 
$\Sigma^{\tmop{reg}}\subsetneq \Sigma^{\tmop{smooth}}$.
\end{Remark}
Under the above assumptions, we can apply the Theorem of Resolution of Singularities for globally 
defined real analytic sets (see e.g.\ \cite{BM}).  As a result,  we conclude 
that there exists a proper analytic 
map $\Phi: N \rightarrow M$, defined by a locally finite sequence of blowing-ups, 
such that:
\begin{enumeratenumeric}
\item $\Phi$ is a diffeomorphism outside $\Phi^{-1}(\Sigma^{\tmop{exc}})$.
\item $D = \Phi^{-1}(\Sigma^{\tmop{exc}})$ is a locally finite union of boundary components 
$$\bigcup_{i \in I}D_i \subset \partial N$$ 
of codimension one.
\item The closure of $\Phi^{-1}(\Sigma^{\tmop{reg}})$ is a smooth submanifold 
$\Omega \subset N$.
\end{enumeratenumeric}
The next result states that, under the above conditions,  the foliation $\vF$ 
pulls-back to a discontinuous foliation in 
$N$ which has a smooth discontinuity locus.
\begin{theorem}\label{theorem-singularlocus}
Let  $\vF$ is a piecewise smooth 1-foliation with analytic discontinuity locus.  
Then, using 
the above notation, there is a piecewise smooth 1-foliation $\vG$ defined on $N$, 
which is related to $\vF$ by $\Phi$, 
and whose discontinuity locus is $\Omega$ .
\end{theorem}
\begin{proof}
We describe separately the local definition of  $\vG$ in points lying in $N\setminus (D\, \cup\, \Omega)$, 
in points lying in  $D \setminus \Omega$, and then in 
points lying in $\Omega$.  
If $q \in N\setminus (D \,\cup\, \Omega)$ then choose a generator $(U,X)$ of 
$\vF$ defined in a sufficiently small 
neighborhood of $p = \Phi(q)$ and decree that
$$
(V,Y) = \big( \Phi^{-1}(V),\Phi^*Y \big)
$$
is a local generator of $\vG$ near $q$.  Notice that this construction defines 
unambiguously the 1-foliation 
on $N \setminus (D \cup \Omega)$, since it is independent of the choice of $(U,X)$.

Let us show now how to extend $\vG$ smoothly to $D \setminus \Omega$. By an induction argument, it suffices to consider the case where $\Phi: N \to M$ is defined by a single blowing-up with center on an analytic submanifold $C \subset M$, and such that $D = \Phi^{-1}(C)$.

Given a point $q \in D \setminus \Omega$, let $p = \Phi(q)$ 
be its image in $C$.  According to the definition of piecewise smooth 1-foliation, 
choose a local multi-generator $(U,\{X_1,\ldots,X_k\})$ of $\vF$ at $p$.  Then, 
considering the disjoint decomposition (\ref{componentsofSigma}), there exists precisely one index
$i \in \{1,\ldots,k\}$, say $i = 1$, such that $W = \Phi^{-1}(U_1)$ is an open neighborhood of $q$.   

Up to reducing these neighborhoods, we choose local trivializing coordinates $(y_1,\ldots,y_n)$ at $W$ and $(x_1,..,x_n)$ at $U_1$, respectively, 
such that $C = \{x_1 = \cdots = x_d = 0\}$ and $D = \{y_1 = 0\}$.  Further, we can assume that, in these coordinates,
the blowing-up map assumes the form 
$$
x_1 = y_1, x_2 = y_1 y_2,\ldots, x_d = y_1 y_d, \quad\text{Êand }\quad \text x_{d+1} = y_{d+1}, \ldots, x_n = y_n.
$$
From the assumption that $X_1$ is non-flat, it follows that the set 
$$
E = \{m \in \Z \; : \; y_1^m\  \Phi^* X_1 \text{Ê  extends smoothly to $\{y_1 = 0\}$} \}
$$
has the form $m_\mini + \N$, for some minimal element $m_\mini \in \Z$.  Indeed, if we expand the vector field $X_1$ as $a_1\frac{\partial}{\partial x_1} + 
\ldots + a_n \frac{\partial}{\partial x_n}$, with $a_i \in C^\infty(U_1)$, then its pull-back under $\Phi$ has the form $\Phi^* X_1 = b_1\frac{\partial}{\partial y_1} + 
\ldots + b_n \frac{\partial}{\partial y_n}$ where 
$$
\begin{array}{l}
b_1 = \Phi^*(a_1), b_2 = \Phi^*\big((a_2 x_1- a_1 x_2)/x_1^2\big), \ldots, b_d = \Phi^*\big((a_d x_1- a_1 x_d)/x_1^2\big), \\
b_{d+1}Ê= \Phi^*(a_{d+1}), 
\ldots, b_nÊ= \Phi^*(a_n)
\end{array}
$$
In particular, $-m_\mini$ is algebraically defined as the minimum of all valuations
$$
\mathrm{val}_{y_1}(\widehat{b}_1),\ldots,  \mathrm{val}_{y_1}(\widehat{b}_n)
$$
where $\widehat{b}$ denotes the formal series expansion of $b \in C^\infty(W)$ in powers of the $y_1$-variable, seen as element of the field $C^\infty(W \cap D)((y_1))$ of formal Laurent series in $y_1$ with coefficients smooth functions in $y_2,\ldots,y_n$ (\footnote{Notice that this minimum is well-defined by the non-flatness assumption.  Furthermore, this shows that 
$m_\mini$ is uniform, i.e. independent of the choice of the point $q \in W\cap D$.}).
Using this we define $Y = y_1^{m_\mini} \Phi^* X_1$ and decree that
$(W,Y)$
is a local generator of $\vG$ at $q$.  

Let us show that this local definition is independent of the choice of the local multi-generator 
$(U,\{X_1,\ldots,X_k\})$. For this, suppose that we choose another  local 
multi-generator $(U',\{X_1',\ldots,X_l'\})$ of $\vF$ at $p$. 
Then, up to a reordering of indices and a restriction to some possibly smaller neighborhood of $p$, 
we can assume that $U_1 = U_1'$ and that 
$$X_1' = \varphi X_1$$
for some smooth function $\varphi$ which is strictly positive in $U$ (we use here the condition 
(\ref{local-compatibilityofgen}) in the definition of a piecewise smooth 1-foliation).  
Taking the pull-back of this relations through $\Phi$, one obtains 
$$\Phi^* X_1'=(\Phi^* \varphi) \, \Phi^* X_1,\quad\text{Êwhere }\quad\Phi^* \varphi \stackrel{\rm def}{=} \varphi\circ\Phi$$
which shows that the set $E$ defined above coincides with the set
 $E' = \{m \in \Z :  y_1^m \Phi^* X_1' \text{Ê  extends smoothly  $\{y_1 = 0\}$} \}$.  
 Consequently, 
$$Y' \stackrel{\rm def}{=} y_1^{m_\mini} \Phi^* X_1 = y_1^{m_\mini} (\Phi^* \varphi) \, \Phi^* X_1^\prime = (\Phi^* \varphi)\,  Y$$
which shows that the foliation $\vG$ is well defined at $q$.\footnote{Notice that this construction is a slight generalization of the blowing-up of local foliated vector fields defined in \cite{DR}.}

\begin{figure}[htbp]
\begin{center}
{ 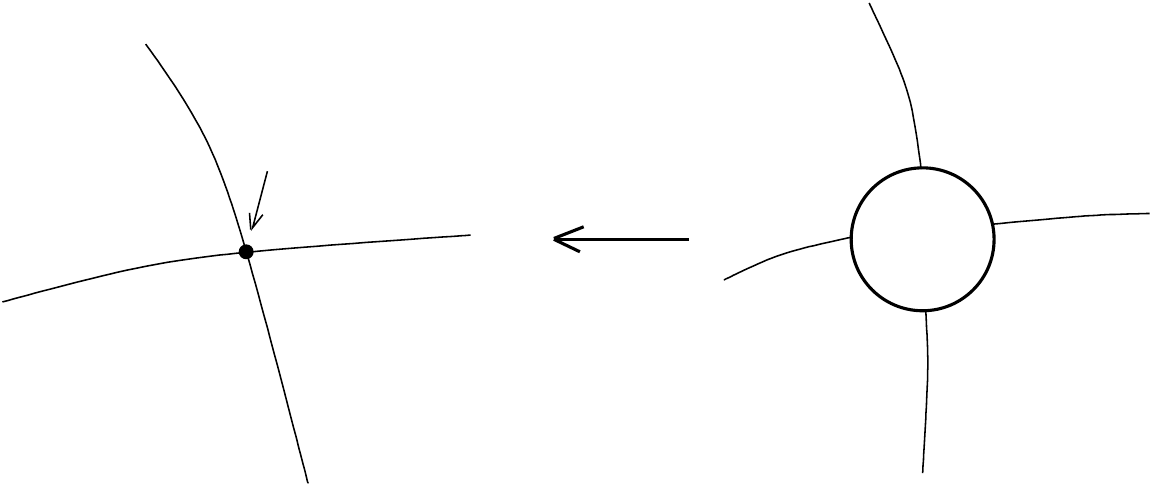}
\caption{Eliminating $\Sigma^{\tmop{exc}}$.}
\label{resolution-fig}
\end{center}
\end{figure}

It remains to construct a local generator of $\vG$ in each point $q \in \Omega$.  
The reasoning is very similar to the previous cases, and is left to the reader.
\end{proof}
Let us show that the previous two Theorems can be combined to give a general 
smoothing procedure which extends Theorem \ref{theorem-smoothlocus} to the case 
where $\Sigma$ is singular. 
\begin{corollary}\label{corr}
Under the assumptions of the Theorem \ref{theorem-singularlocus}, suppose further that the 
discontinuity locus of $\vF$ has codimension one.  Then,  $\vF$ is blow-up smoothable.  
\end{corollary}

\section{Regularization and sliding dynamics for piecewise smooth foliations} \label{s2}
One disadvantage of the smoothing procedure defined in the previous subsection is that 
it does not allow to define the so-called {\em sliding dynamics }along the discontinuity set.  

Our present goal is to define such sliding dynamics as some sort of limit of the dynamics of 
a nearby smooth 1-foliations. This 
leads us to introduce the notion of regularization.  Later on, we shall see that the blow-up 
smoothing and the regularization can be combined in a fruitful way. 
 
Let $\vF$ be a discontinuous 1-foliation on a manifold $M$, with discontinuity locus $\Sigma$.  
A {\em regularization of $\vF$  (with $p$-parameters) }is a discontinuous 1-foliation $\vF^\reg$ 
defined in the product manifold
$$
M \times ((\R^+)^p,0)
$$
(\footnote{We use the notation $((\R^+)^p,0)$ to indicate in abridged form some open 
neighborhood of the origin in  $(\R^+)^p$.}) which satisfies the three following conditions:
\begin{enumeratenumeric}
\item $\vF^\reg$ is tangent to the fibers of the canonical projection 
$$\pi: M \times ((\R^+)^p,0) \rightarrow ((\R^+)^p,0)$$
\item The restriction $\vF^\reg_0$ of $\vF^\reg$ to the fiber $\pi^{-1}(0)$ coincides with $\vF$,
\item The discontinuity locus $\Sigma^\reg$ of $\vF^\reg$ is a subset of  $\Sigma \times \big\{\prod_i \eps_i = 0\big\}$, 
where $(\eps_1,..,\eps_p)$ are the coordinates in $((\R^+)^p,0)$.
\end{enumeratenumeric}

The last condition implies that, for each $\eps \in ((\R^+)^p,0)$ such that $\prod_i \eps_i  \ne 0$, 
the restriction $\vF^\reg_\eps$ of $\vF^\reg$ to the fiber $\pi^{-1}(\eps)$ is a smooth 1-foliation.  
Furthermore, by the smoothness assumption,
$$
\lim_{\eps \rightarrow 0}\vF^\reg_\eps = \vF^\reg_0
$$
uniformly (in the $C^\infty$ topology) on each compact subset of $M \setminus \Sigma$ 
(\footnote{More precisely, given a point $q \in M\setminus \Sigma$, there 
exists an open neighborhood $U$ of $q$ and a $p$-parameter family of smooth vector 
field $X_\eps$ defined on $U \times ((\R^+)^p,0)$ (and depending smoothly on $\eps$) 
such that $X_\eps$ is  a local generator of $\vF^\reg_\eps|_U$ for each $\eps$.}).

\begin{Example}
In \cite{H} section 1.4, H\" ormander constructs a regularization by convolution. For simplicity, let us 
assume that $\vF$ is defined in $\cR^n$ by a smooth vector field $X$ which extends as 
a locally bounded measurable function to the discontinuity set $\Sigma$.  
Given a function $0 \le \chi \in C_0^\infty(\cR^n)$ such that $\int \chi(y) dy = 1$, we define
$$
X_\eps(x) = \int_{\R^n} X(x-\eps y) \chi(y) dy
$$
Then, it is easy to see that $X_\eps$ is a smooth vector field for each $\eps > 0$ and 
that the one-parameter family of 1-foliations $\vF_\eps$ defined by these vector fields 
is a regularization of $\vF$.
\end{Example}

One disadvantage of this regularization by convolution is that some important features 
of  the dynamics of $\vF$ which appears outside the discontinuity locus can be destroyed 
by small perturbations, and thus not be seen in $\vF_\eps$.  For instance, the saddle connection illustrated in figure~\ref{break-saddle-fig}
would be broken by a generic choice of convolution kernel $\chi$ (although it lies outside the discontinuity locus).

\begin{figure}[htbp]
\begin{center}
{ 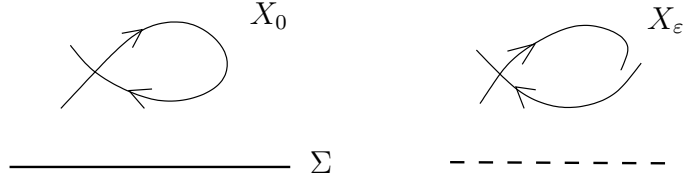}
\caption{Saddle connection for $X_0$ is broken by the regularization by convolution.}
\label{break-saddle-fig}
\end{center}
\end{figure}

In the next subsection, we will describe two regularization methods which keep $\vF$ 
unchanged outside  $O(\eps)$-neighrborhoods of the discontinuity set.  As such, 
we expect to see the full dynamics of $\vF$ outside $\Sigma$ to be reflected at 
$\vF_\eps$, for each sufficiently small $\eps$.   

\subsection{Sotomayor-Teixeira regularization and its generalizations}\label{subsection-ST}
Suppose that the discontinuity locus of $\vF$ is a smooth submanifold 
$\Sigma \subset M$ of codimension one and that we fix the following data:

\begin{enumeratenumeric}
\item A tubular neighborhood map $f: N\Sigma \rightarrow M$, which maps 
the normal bundle $N\Sigma$ diffeomorphically to an open neighborhood $W=f(N\Sigma)$ of $\Sigma$.
\item A smoothly varying metric $|\cdot|$ on the fibers of the bundle 
$N\Sigma \rightarrow \Sigma$ (such that $|p|=0$ iff $p\in \Sigma$).
\item A {\em monotone transition function }$\phi:\cR^ \rightarrow [-1,1]$.
\end{enumeratenumeric}

Using the map $f$, we pull-back $\vF$ to a discontinuous 1-foliation $\vG$ on 
the normal bundle $N\Sigma$, with discontinuity locus given by the zero section 
$\Sigma \subset N\Sigma$. 

Without loss of generality, we can assume that $N\Sigma$ is covered by local 
trivialization charts where the bundle map assumes the form
$$
\begin{array}{rcl}
V \times \cR  &\longrightarrow& V \\
(x,y) & \longmapsto & x
\end{array}
$$ 
for some open set $V\subset \Sigma$, and that $\vF$ has a local multi-generator 
of the form $(V\times\cR,\{X_+,X_-\})$, where $X_+$ (resp. $X_-$) is a smooth vector 
field in $V\times \cR$ which generate $\vG$ on 
$U_+ = \{y >0\}$ (resp. $U_- = \{y < 0\}$).  Furthermore, we can assume that the norm 
on the fibers of $N\Sigma$ is simply the absolute value $|y|$ on $\R$.

For each $\eps > 0$, we now define a smooth vector field $X_\eps$ in $V\times\cR$ as follows
$$
X_\eps \stackrel{\mathrm{def}}{=} \frac{1}{2}\left(1 +  \phi\Big(\frac{y}{\eps}\Big)\, \right) X_+ + \frac{1}{2}\left(1 - \phi\Big(\frac{y}{\eps}\Big)\, \right) X_-
$$
Notice that, by construction
$$
X_\eps(x,y) = 
\begin{cases}
X_+ (x,y)& \text{Êif }y  \ge \eps,\\
X_- (x,y)& \text{Êif }y \le -\eps,\\
\end{cases}
$$
Moreover, if we choose another multi-generator of $\vG$, say $(V\times\cR,\{Y_+,Y_-\})$ 
then it follows from the condition (\ref{local-compatibilityofgen}) in the definition of piecewise 
smooth 1-foliation that $Y_+ = \varphi X_+$ and $Y_- = \varphi X_-$, for some strictly positive 
smooth function $\varphi$.  Therefore, if we define a family
$Y_\eps$ exactly as above but replacing $X_\pm$ by $Y_\pm$, we conclude that 
$$Y_\eps = \varphi X_\eps, \quad \forall \eps > 0$$
In other words, the $X_\eps$ and $Y_\eps$ define precisely the same smooth 1-foliation 
in the domain $V \times \cR$.  

By considering an open covering of $N\Sigma$ by these local trivializations, one defines, 
for each $\eps > 0$, a smooth foliation $\vG_\eps$. By construction, such foliation  coincides
 which the original foliation $\vG$ outside the region $\{p \in N\Sigma : |p| < \eps\}$. 

The {\em Sotomayor-Teixeira regularization }of $\vF$ is the discontinuous 1-foliation $\vF^\reg$ 
defined in the product space $M \times (\R^+,0)$ as follows:
For $\eps = 0$, we let $\vF^\reg_0 = \vF$.  For $\eps > 0$, we define the  foliation $\vF^\reg_\eps$ in $M$ by
$$
\vF^\reg_\eps = 
\begin{cases}
\vF&\text{on }M \setminus W\\
f_* \vG_\eps&\text{on }W
\end{cases} 
$$
It follows from the remark made at the previous paragraph that this defines a globally 
smooth 1-foliation in $M$.  It is easy to verify that the conditions 1.-3. of the definition 
of a regularization are satisfied.

\begin{figure}[htbp]
\begin{center}
{ 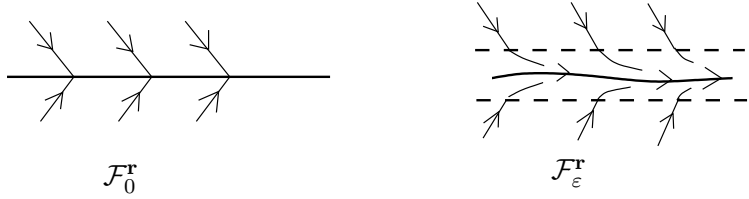}
\caption{The Sotomayor-Teixeira regularization.}
\label{reg-fig}
\end{center}
\end{figure}

More generally, under the same assumptions of the previous example, we can define regularization of $\vF$ 
by dropping the assumption of monotonicity and $x$-independence of the transition function.  
Namely, by replacing the choice of function $\phi$ in item 3. by the choice of a smooth function 
\begin{equation}\label{psi-function}
\psi : \Sigma \times \cR_+ \rightarrow [-1,1]
\end{equation}
such that $\psi(x,t) = -1$ if $t \le -1$ and $\psi(x,t) = 1$ if $t \ge 1$.  Correspondingly, we  
replace the expression of $X_\eps$ given above by
\begin{equation}\label{def-transitionreg}
X_\eps \stackrel{\mathrm{def}}{=} \frac{1}{2}\left(1 +  \psi\Big(x,\frac{y}{\eps}\Big)\, \right) X_+ + \frac{1}{2}\left(1 - \psi\Big(x,\frac{y}{\eps}\Big)\, \right) X_-
\end{equation}
All the remaining steps in the construction remain the same.  The resulting regularization 
will be called a \emph{regularization of transition type}.

\subsection{Double regularization of the cross}\label{subsect-cross}
Let us show a situation where it is natural to consider a multi-parameter regularization.  
 Consider a discontinuous 1-foliation $\vF$ in $\R^3$ with discontinuity locus $\Sigma = \{ xy = 0\}$
  (like in Example~\ref{example-3}). In other words, $\vF$ is defined by four smooth vector fields 
$X_{\pm,\pm}$, 
where the first and the second $\pm$ sign correspond respectively to the sign of the $x$ and $y$ 
coordinates. In other words, each $X_{\pm,\pm}$ is a generator of $\vF$ in one of the four quadrants
$U_{\pm,\pm} = \{(x,y,z) : \sgn(x) = \pm, \sgn(y) = \pm\}$.  

Choosing monotone transitions functions $\phi,\psi$ as above, we consider the two-parameter 
family of smooth vector fields
$$
Y_{\eps,\eta} \stackrel{\mathrm{def}}{=} 
\sum_{\alpha,\beta \in \{+,-\}} \left( 1 + \alpha \phi\Big( \frac{x}{\eps} \Big) \right) \left( 1 + \beta \psi\Big( \frac{y}{\eta} \big) \right) X_{\alpha,\beta}
$$
defined for $\eps,\eta > 0$.  Similarly, we define the two one-parameter families of 
discontinuous vector fields
$$
\begin{array} {ccc}
Z_{\pm,\eta} &\stackrel{\mathrm{def}}{=} & 
\displaystyle \sum_{\beta \in \{+,-\}}  \left( 1 + \beta \psi\Big( \frac{y}{\eta} \big) \right) X_{\pm,\beta}\\
W_{\eps, \pm} &\stackrel{\mathrm{def}}{=} &
\displaystyle \sum_{\alpha \in \{+,-\}} \left( 1 + \alpha \phi\Big( \frac{x}{\eps} \Big) \right)  X_{\alpha,\pm}
\end{array}
$$
defined respectively for $\eta >0$ and $\eps > 0$.  Notice that the discontinuity locus of 
$Z_{\pm,\eta}$ and $W_{\pm,\eta}$ is given respectively by 
$\{x = 0\}$ and $\{y = 0\}$. 

The {\em double-regularization }of $\vF$ is the discontinuous 1-foliation $\vF^\reg$ defined in
 the product space $\R^3 \times ((\cR^+)^2,0)$ as follows.  For each 
parameter value $\eps,\eta \ge 0$, the foliation restricted to fiber $\pi^{-1}(\eps,\eta)$ is
 generated by a discontinuous vector field $K_{\eps,\eta}$ chosen as follows
$$K_{\eps,\eta} = 
\begin{cases}
X_{\pm,\pm} & \text{Êif $\eps = 0$ or $\eta = 0$,}\\
Z_{\pm,\eta}  & \text{Êif $\eps = 0$ and $\eta > 0$},\\
 W_{\eps,\pm}  & \text{Êif $\eps > 0$ and $\eta = 0$},\\
Y_{\eps,\eta} & \text{Êif $\eps > 0$ and $\eta > 0$},
\end{cases}
$$

\begin{figure}[htbp]
\begin{center}
{ 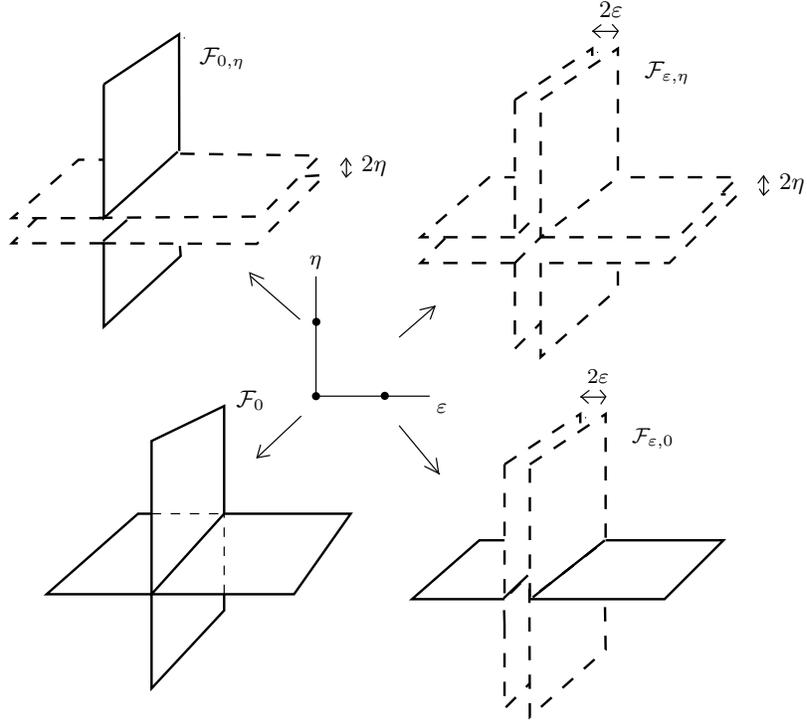}
\caption{The double regularization and the discontinuity locus}
\label{double-reg-fig}
\end{center}
\end{figure}

More generally, assuming that a discontinuous foliation $\vF$ in $\R^n$ has a discontinuity locus 
given by the union of $p$ coordinate hyperplanes, say
$$\Sigma = \Big\{\prod_{i=1}^p x_i = 0\Big\}$$
we can define $p$-parameter regularization of $\vF$ by an easy generalization of the above construction.

\subsection{Sliding regions}
Let $\vF$ be a discontinuous $1$-foliation defined on a manifold $M$ and with discontinuity locus $\Sigma$.   
Given a $p$-parameter regularization $\vF^\reg$ of $\vF$, our present goal is to define a subset 
$\Slide(\vF^\reg)$ of $\Sigma$ where it will be reasonable to study a {\em limit dynamics with
 respect to such given regularization}.

Our definition is local.  We will say that point $p \in \Sigma$ is a {\em point of sliding for $\vF^\reg$ }
if there exists an open neighborhood $U \subset M$ of $p$ and a family of smooth manifolds
$$
S_\eps \subset U
$$
defined for all $\eps \in ((\R^\star)^p,0)$ such that:

\begin{enumeratenumeric}
\item For each $\eps$, $S_\eps$ is invariant by the restriction of $\vF^\reg_\eps$ to $U$.
\item For each compact subset $K \subset U$, the sequence $S_\eps \cap K$ converges to 
$\Sigma \cap K$ as $\eps$ goes to zero in some given Hausdorff metric $d_H$ on compact sets of 
$M$. (\footnote{Obviously, this metric depends on the choice of a Riemannian metric on $M$, 
but the convergence condition is independent of the choice of this metric.})
\end{enumeratenumeric}

\begin{figure}[htbp]
\begin{center}
{ 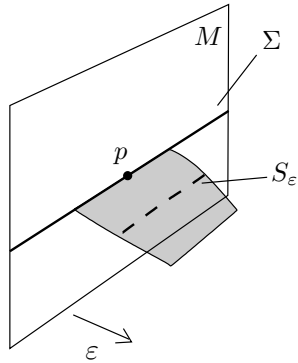}
\caption{The sliding region.}
\label{double-reg-fig}
\end{center}
\end{figure}

The set of sliding points for $\vF^\reg$ is a relatively open subset of $\Sigma$, which we denote by 
$\Slide(\vF^\reg)$. 

\begin{Example}
Consider the Sotomayor-Teixeira regularization of the discontinuous foliation described in Example \ref{example-1}. 
 Then, an easy computation with the expression of $X_\eps$ defined in the previous subsection 
 (and taking $|y|$ to be the usual absolute value) gives
$$
X_\eps = \frac{\partial}{\partial x} - \phi\left( \frac{y}{\eps}\right) \frac{\partial}{\partial y}.
$$
Let $t_0 \in (-1,1)$ be the zero of $\phi$ (which is unique by the monotonicity hypothesis on $\phi$).  
Then, the family of one-dimensional manifolds 
$$
S_{\eps} = \{y = t_0 \eps\}
$$
satisfies  the above conditions 1. and 2. locally at each point of $\Sigma$.  Consequently, $\Slide(\vF^\reg) = \Sigma$.
\end{Example}

\begin{Example}
Consider the discontinuous foliation in $\R^2$ defined by the vector field
$$
X = \frac{\partial}{\partial x} + \big( (x+1) + sgn(y)(x-1) \big) \frac{\partial}{\partial y}.
$$
Then, the Sotomayor-Teixeira regularization is defined by the vector field
$$
X_\eps = \frac{\partial}{\partial x} - \big( (x+1) + \phi\left( \frac{y}{\eps}\right) (x-1) \big)\frac{\partial}{\partial y}.
$$

\begin{figure}[htbp]
\begin{center}
{ 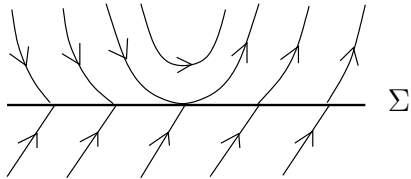}
\caption{Fold point.}
\label{fold-point-fig}
\end{center}
\end{figure}

Therefore for each $x \in \cR$, the coefficient of the $\partial/\partial y$ component of $X_\eps$ vanishes if and only if
$$
y = \eps t_x
$$
where $t_x$ is  a solution of the equation $\phi(t_x) =   (x+1)/(1-x)$. By the assumptions on $\phi$, 
this equation has a solution (which is necessarily unique) if and only if $x \le 0$.  As we shall prove in the next section, it follows that 
$\Slide(\vF^\reg) = \Sigma \cap \{x < 0\}$.
\end{Example}

\begin{Example}
Let us consider the same discontinuous vector field of the previous Example but now use a different regularization.  
Namely, we let $\vF^\reg$ be a regularization of transition type, with transition function $\psi$ 
having a graph as illustrated in the figure below.

\begin{figure}[htbp]
\begin{center}
{ 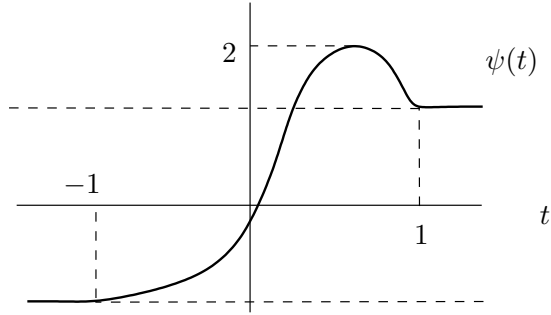}
\caption{A non-monotone transition function.}
\label{double-reg-fig}
\end{center}
\end{figure}

As a consequence of the results of the next section, we have $\Slide(\vF^\reg) = \Sigma \cap \{x < \frac{1}{3} \}$ 
(because $\frac{1/3 + 1}{1-1/3 } = 2$).
\end{Example}
\begin{Remark}
(Stratified Sliding for analytic discontinuity locus)  Assume that the discontinuity locus 
$\Sigma$ is an analytic subset, of dimension $d$.  Then, we can define a more refined notion 
of sliding by considering different strata of $\Sigma$. 

More precisely,  using the decomposition $
\Sigma^0  \subset \Sigma^1 \subset \cdots \subset \Sigma^d = \Sigma
$ defined in (\ref{stratification}), we say that point $p \in \Sigma^k \setminus \Sigma^{k-1}$ is a 
{\em stratified point of sliding }for $\vF^\reg$ if the conditions 1.\ and 2.\ of the above 
definition holds, when we replace the convergence condition in 2.\ by
$$d_H(S_\eps \cap K, \Sigma^k \cap K) \rightarrow 0$$
as $\eps$ goes to zero.   The set of all points $\Sigma^k \setminus \Sigma^{k-1}$ satisfying 
the above condition  is called {\em sliding region of dimension $k$}, and denoted by $\Slide^k(\vF^\reg)$.
\end{Remark}

\begin{Example}
Let us apply the double regularization described in subsection~\ref{subsect-cross} to the 
discontinuous 1-foliation defined in Example \ref{example-3}. An easy computation shows that 
the regularized vector field is given by
$$
X_{\eps,\eta} = - \phi\left(\frac{x}{\eps}\right) \frac{\partial}{\partial x}Ê-  \psi\left(\frac{y}{\eta}\right) \frac{\partial}{\partial y} + \frac{\partial}{\partial z} 
$$
where $\phi,\psi$ are monotone transition functions.  

\begin{figure}[htbp]
\begin{center}
{ 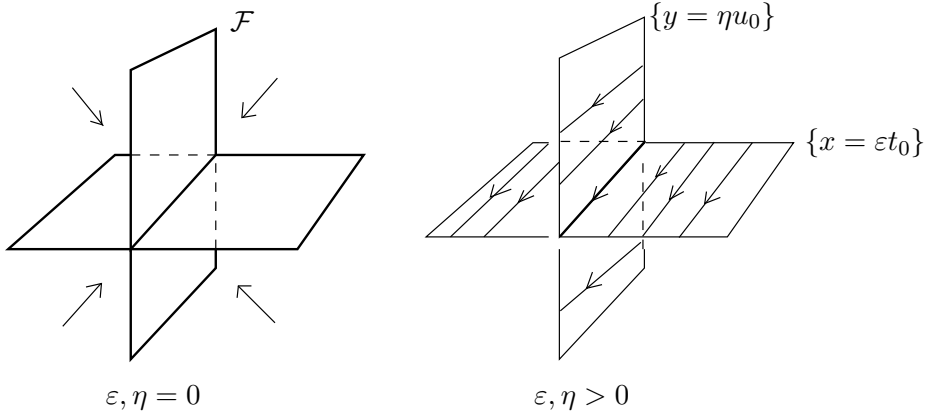}
\caption{A stratified sliding.}
\label{double-reg-slide-fig}
\end{center}
\end{figure}

Consider the one dimensional stratum $\Sigma^1 \subset \Sigma$ given by $\Sigma_1 = \mbox{axis}\, z$.  
We claim that each point of $\Sigma^1$ is a stratified point of sliding.    Indeed, 
if we denote by $t_0,u_0 \in (-1,1)$ the unique roots of $\phi(t) = 0$ and $\psi(u) = 0$ respectively, then the one-dimensional curve
$$
S_{\eps,\eta} = \{x = \eps t_0,\ y = \eta u_0 \}
$$
is invariant by $X_{\eps,\eta}$ and clearly converges to $\Sigma^1$ as $\eps,\eta$ converges to zero.
\end{Example}

\section{Regularizations of transition type: blowing-up and conditions for sliding}\label{s3}
In this section, we consider piecewise  smooth 1-foliations whose discontinuity set 
is a smooth submanifold of codimension one.  
Our main goal is to describe conditions which guarantee that a point belongs to the 
sliding region of given a regularization of transition type.   

To fix the notation, we choose a piecewise smooth 1-foliation $\vF$ defined in a manifold $M$, 
and whose discontinuity locus is a smooth submanifold $\Sigma \subset M$ of codimension one.  
According to the definition in subsection~\ref{subsect-piecewise-smooth}, at 
each point $p \in \Sigma$ 
we can choose local coordinates $(x_1,\ldots,x_{n-1},y)$ and two smooth vector fields $X_+$ and $X_-$ such that 
$\Sigma = \{ y = 0\}$ and $X_+$ and $X_-$ are generators of $\vF$ on the sets $\{y > 0\}$ and $\{y < 0\}$, respectively.  

First of all, we prove a result which will allow us to use the the theory of smooth 
dynamical systems to study such regularization.

\begin{theorem}\label{theorem-smoothableST}
Let $\vF^\reg$ be a regularization of transition type of $\vF$.  Then, $\vF^\reg$ is blow-up smoothable.
\end{theorem}
\begin{proof}
We will show that a single blowing-up suffices to obtain a smooth foliation.  
More precisely, consider the blowing up 
$$\Phi: N \rightarrow M \times (\R^+,0)$$
with center on $\Sigma$.  We claim that there exists a smooth foliation in $\vG$ 
in $N$ which is related to $\vF^\reg$ by $\Phi$. 

\begin{figure}[htbp]
\begin{center}
{ 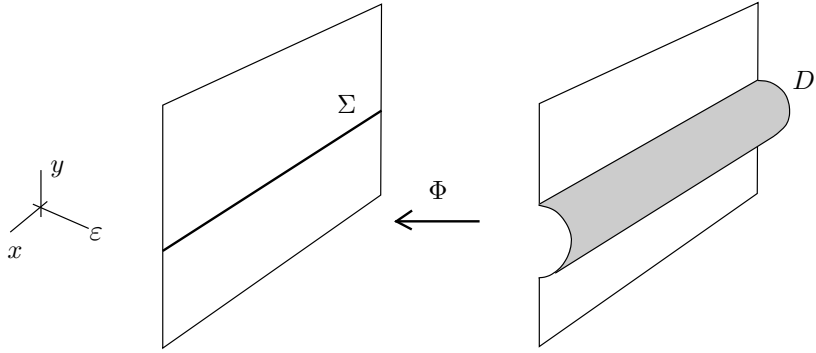}
\caption{Blowing-up the regularization.}
\label{double-reg-slide-fig}
\end{center}
\end{figure}

To prove this,  we use the trivialization
of $\Sigma$ given by the local charts $(x,y) \in V \times \R$ described in subsection \ref{subsection-ST} 
and the expression for $X_\eps$ defined at the end of that subsection.  Using these coordinates, 
the blowing-up map can be written (up to restriction to an appropriate subdomain) as
$$
\begin{array}{ccc}
V \times \cS^1 \times \cR^+ & \longrightarrow & V \times \R \times \R \\
x,(\theta,r) & \longmapsto & x,y=r \sin(\theta), \eps = r \cos(\theta)
\end{array}.
$$
In order to make the computations easier, it is better to cover the domain in $\cS^1 \times \cR^+$ 
by three directional charts, with domains 
$E = \{\theta \not\equiv 0 (\mathrm{mod}\, \pi \Z)\}$ and $F_\pm = \{\theta \not\equiv \pm \pi/2 (\mathrm{mod}\, 2\pi \Z)\}$.  
In these charts, the blowing-up map assumes 
respectively the form
$$
\begin{array}{c}
\begin{array}{ccc}
V \times \R \times \R^+ & \longrightarrow & V \times \R \times \R \\
x,(\bar{y},\bar{\eps}) & \longmapsto & x,y= \bar{\eps}\bar{y}, \eps = \bar{\eps}
\end{array} ,\\ \\
\begin{array}{ccc}
V \times \R^+ \times \R^+ & \longrightarrow & V \times \R \times \R \\
x,(\tilde{y},\tilde{\eps}) & \longmapsto & x,y=\pm \tilde{y}, \eps =  \tilde{y} \tilde{\eps}
\end{array}.
\end{array}$$
Let us compute the pull-back of $X_\eps$ in each one of these charts.  

In the $E$-chart, we have the following transformation rules for the basis vectors of $T(M\times \R)$:
$$
\frac{\partial}{\partial x_i} = \frac{\partial}{\partial x_i},Ê\quad
\frac{\partial}{\partial y} = \frac{1}{\bar{\eps}} \frac{\partial}{ \partial \bar{y}} ,\quad
\frac{\partial}{\partial \eps} = \frac{\partial}{ \partial \bar{\eps}}  - \frac{\bar{y}}{\bar{\eps}} \frac{\partial}{\partial \bar{y}} .
$$
Thus, for instance the vector field $f(x,y,\eps)\frac{\partial}{\partial y}$ is mapped to 
$\frac{f(x,\bar{y}\bar{\eps},\bar{\eps})}{\bar{\eps}} \frac{\partial}{\partial \bar{y}}$ and son on.
Therefore, using the expression in (\ref{def-transitionreg}), we obtain
$$
\Phi^* X_\eps = \frac{1}{2}\left(1 +  \psi\big(x,\bar{y}\big)\; \right) \Phi^* X_+ + \frac{1}{2}\left(1 - \psi\big(x,\bar{y}\big)\; \right) \Phi^* X_-.
$$
Therefore, from the above transformation rules, it is clear that the vector field
\begin{equation}\label{vector-fieldY}
Y \stackrel{def}{=} \bar{\eps}\, \Phi^* X_\eps
\end{equation}
has a smooth extension to the exceptional divisor $\{\bar{\eps} = 0\}$.  
We take $Y$ to be a local generator of $\vG$ on this domain.

Similarly, in the  $F_\pm$-chart, we have the following transformation rules 
$$
\frac{\partial}{\partial x_i} = \frac{\partial}{\partial x_i},Ê\quad
\frac{\partial}{\partial \eps} = \pm \frac{1}{\tilde{y}} \frac{\partial}{ \partial \tilde{\eps}} ,\quad
\frac{\partial}{\partial y} = \pm \left(  \frac{\partial}{ \partial \tilde{y}}  - \frac{\tilde{\eps}}{\tilde{y}} \frac{\partial}{\partial \tilde{\eps}} \right).
$$
And the pull-back of $X_\eps$ in this chart has the form
$$
\Phi^* X_\eps = \frac{1}{2}\left(1 +  \psi\Big(x,\pm \frac{1}{\tilde{\eps}}\Big)\; \right) \Phi^* X_+ + \frac{1}{2}\left(1 - \psi\Big(x,\pm \frac{1}{\tilde{\eps}}\Big)\; \right) \Phi^* X_-.
$$
Notice that, for all $0 < \tilde{\eps} \le 1$, one has $\psi\Big(x,\pm \frac{1}{\tilde{\eps}}\Big) \equiv \pm 1$
 identically, and we can extend this function smoothly to $\tilde{\eps} = 0$ as being equal to $\pm 1$, according to the domain. 
  Therefore,  similarly to the  previous case, the vector field 
$$
Z \stackrel{def}{=} \tilde{y}\, \Phi^* X_\eps
$$
has a smooth extension to the exceptional divisor $\{\tilde{y} = 0\}$, and we choose it as a generator of $\vG$ 
on the corresponding domain.  This concludes the proof.
\end{proof}

Now, we will study the sliding regions.  The criterion that we are going to describe needs one additional definition: 
 Using the notation introduced above, the {\em height function of $\vF^\reg$ }is the smooth function $h^\reg$ 
 with domain $(x,t) \in \Sigma\times \R$ defined by
$$
h^\reg = 
 \psi\, \mathcal{L}_{(X_+ - X_-)}(y) + {\mathcal{L}_{(X_+ + X_-)}(y)}
$$
where $\psi(x,t)$ is the transition function and $\mathcal{L}_X(f)$ denotes the Lie derivative of a function $f$ 
with respect to a vector field $X$. We remark that that the Lie derivative of $X_+ - X_-$ and $X_+ + X_-$ 
needs to be evaluated only at points 
of $\Sigma$.  

More explicitly, if we write $X_+$ and $X_-$ in terms of the local trivializing coordinates $(x,y)$ described 
above as  
\begin{equation}\label{expressions-Xpm}
X_\pm  = a_\pm \frac{\partial}{\partial y}Ê+ \sum_{i=1}^{n-1} b_{i,\pm} \frac{\partial}{\partial x_i}
\end{equation}
(for some smooth functions $a_\pm$ and $b_{i,\pm}$) then the height function is given by
$$
h^\reg(x,t) =  \psi(x,t)\, \Big(a_+(x,0) - a_-(x,0)\Big) +  \Big(a_+(x,0) + a_-(x,0)\Big).
$$

Notice that the function $h^\reg$ is independent of the choice of local coordinates $(x,y)$ and 
local generators $(X_+,X_-)$ up to multiplication by a strictly positive function.  More precisely,  
if we replace the local generators $(X_+,X_-)$ by local generators $(Y_+,Y_-)$ such 
that $Y_\pm = \varphi X_\pm$ then $h^\reg$ is transformed to
$\varphi h^\reg$.

Based on the height function, we define the following subsets 
in $\Sigma \times \R$:
$$
\begin{array}{rcl}
\displaystyle \vZ^\reg &=& \big\{  h^\reg(x,t) = 0 \big\},\quad 
\vW^\reg = \Big\{  \displaystyle \frac{\partial h^\reg}{\partial t}(x,t) \ne 0\Big\},\;\text{Êand }\\
\NH^\reg &=& \vZ^\reg \cap \vW^\reg.
\end{array}
$$
The main result of this section can now be stated as follows:

\begin{theorem}\label{theorem-slide}
Let $\vF^\reg$ be a regularization of transition type of $\vF$, defined by a transition function $\psi$ 
as above.  Then, 
$$
\pi(\NH^\reg) \subset \Slide(\vF^\reg) \subset \pi(\vZ^\reg).
$$
where $\pi:\Sigma \times \R \rightarrow \Sigma$ is the canonical projection.
\end{theorem}

\begin{proof}
Let us compute in more details the expression of the blowing-up of $X_\eps$ in the $E$ chart described in 
the proof of Theorem~\ref{theorem-smoothableST}.  If we write the transformed vector field $Y$ in the form 
\begin{equation}\label{vectorfield-Y2}
Y = \alpha\  \frac{\partial}{\partial \bar{y}}Ê+ \bar{\eps}\; \sum_{i=1}^{n-1} \beta_{i}\  \frac{\partial}{\partial x_i} 
\end{equation}
then, using the expansions of $X_+$ and $X_-$ given in (\ref{expressions-Xpm}), we conclude that, for $i = 1,...,n-1$,
$$
 \beta_i = \displaystyle \frac{1}{2}  \Big(  \psi\, (b_{i,+} - b_{i,-}) +  (b_{i,+} + b_{i,-}) \Big),
$$
and 
$$
\alpha = \displaystyle \frac{1}{2} \Big(  \psi\, (a_+ - a_-) +  (a_+ + a_-) \Big)
$$ 
where $\psi$ is the transition function evaluated at $(x,t) = (x,\bar{y})$ and all functions $a_\pm$ and $b_{i,\pm}$ 
are computed by replacing the variable $y$ by $\bar{\eps} \bar{y} $.  Notice that the restriction 
$Y|_D$ of $Y$ to the exceptional divisor $D = \{\bar{\eps} = 0\}$ is simply given by
\begin{equation}\label{YonD}
Y|_D = \frac{1}{2} h^\reg(x,\bar{y})\  \frac{\partial}{\partial \bar{y}}Ê
\end{equation}
where $h^\reg$ is the height function defined above.

Suppose now that the coordinates $(x,y,\eps)$ are centered in a point $p \in \Sigma \times \{0\}$ 
lying in $\pi(\NH^\reg)$.  Then, 
it follows from the above definition of the sets $\vZ^\reg$ and $\vW^\reg$ that
there exists a $\bar{y}_0 \in \cR$ lying in the open interval $(-1,1)$ such that 
$$ 
h^\reg(0,\bar{y}_0) = 0, \quad\text{Êand }\quad \frac{\partial h^\reg}{\partial y}(0,\bar{y}_0) \ne 0.
$$
Looking at the expression of $Y|_D$ given above, it follows from the implicit function theorem that the point 
$q = (0,\bar{y}_0,0) \in \Phi^{-1}(p)$ lies in a locally defined smooth codimension one submanifold $H_0$ 
contained in the divisor $D = \{\bar{\eps} = 0\}$ such that 
\begin{enumerate} 
\item Each point of $H_0$ is an equilibrium point of $Y|_D$.
\item $H_0$ is a normally hyperbolic invariant submanifold of $Y|_D$. 
\end{enumerate}
From  Fenichel  theory \cite{F} it follows that there exists a local smooth manifold $W \subset N$ of 
codimension one defined near $p$ which is invariant by the flow of $Y$ and such that $W\cap D = H_0$.  

Let $S = \Phi(W)$.  Then, it follows that, for each sufficiently small $\eps > 0$, the set 
$S_\eps = S \cap \pi^{-1}(\eps)$ is an invariant submanifold of $\vF^\reg_\eps$ and $S_\eps$ 
accumulates on $\Phi(H_0)$ as $\eps$ goes to zero.  Therefore, 
$p \in \Slide(\vF^\reg)$.

\begin{figure}[htbp]
\begin{center}
{ 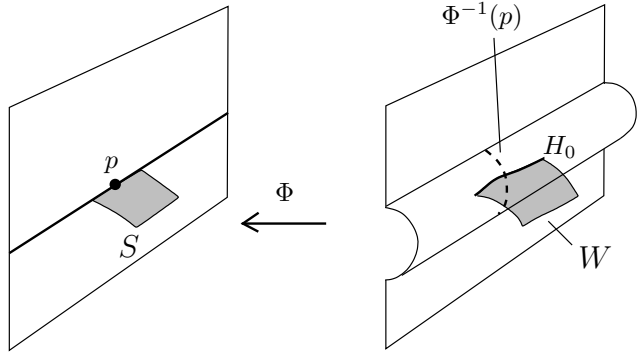}
\caption{The normally hyperbolic manifold $H_0$.}
\label{double-reg-slide-fig}
\end{center}
\end{figure}

We have just proved that $\pi(\NH^\reg) \subset \Slide(\vF^\reg)$.  We postpone the proof of the inclusion 
$\Slide(\vF^\reg) \subset \pi(\vZ^\reg)$ to the end of this section.
\end{proof}
Let us now describe the behavior of a regularization in the complement of the sliding set.  
For this, we introduce the so-called sewing region.  

Keeping the above notation, we will say that a point $p \in \Sigma$ is a \emph{point of sewing }for the regularization 
$\vF^\reg$ if there exists an open neighborhood $U \subset M$ of $p$ and local coordinates $(x,y)$ 
defined in $U$ such that 
\begin{enumeratenumeric}
\item $\Sigma = \{y = 0\}$ and,
\item For each sufficiently small $\eps > 0$, the \emph{vertical vector field } $\frac{\partial}{\partial y}$ 
is a generator of $\vF_\eps^\reg$ in $U$.
\end{enumeratenumeric}
We will denote the set of all sewing points by $\Sew(\vF^\reg)$.
\begin{Remark}\label{intersection-slidesew}
Notice that the intersection of the regions $\Sew(\vF^\reg)$ and $\Slide(\vF^\reg)$ is empty. Indeed, 
if $p$ lies in $\Sew(\vF^\reg)$ then in the coordinates $(x,y)$ described above, each smooth 
manifold $S_\eps$ which is invariant by $\vF_\eps^\reg$ should  have necessarily the form
$$S_\eps = \{f_\eps(x) = 0\}$$
for some smooth function $f_\eps$ which is independent of the $y$ variable. In particular, 
$S_\eps$ cannot tend to the discontinuity locus $\Sigma = \{y = 0\}$ as $\eps$ goes to zero.  
Therefore $p \not\in \Slide(\vF^\reg)$. 
\end{Remark}

\begin{figure}[htbp]
\begin{center}
{ 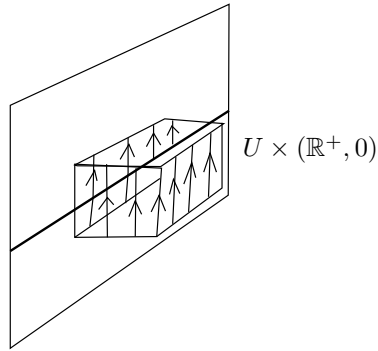}
\caption{A sewing region.}
\label{sewing-region-fig}
\end{center}
\end{figure}

\begin{theorem}\label{theorem-sewingreg}
Let $\vF^\reg$ be a regularization of transition type of $\vF$, defined by a transition function $\psi$.  Then, 
$$
\pi\big(\vZ^\reg \big)^\complement \subset \Sew(\vF^\reg)
$$
where $\pi(\vZ^\reg)^\complement$ denotes the complement of $\pi(\vZ^\reg)$ in $\Sigma$.
\end{theorem}

\begin{proof}
Suppose that the coordinates $(x,y,\eps)$ described in the proof of Theorem~\ref{theorem-smoothableST} 
are centered in a point $p \in \Sigma\times \{0\}$ 
which lies in $\pi\big(\vZ^\reg \big)^\complement$.  We claim that there exists a constant $\mu > 0$ 
and an open neighborhood 
$U \subset M \times (\cR^+,0)$ of $p$ such that the function 
$$g = \mathcal{L}_{X_\eps}(y)$$
(which is defined only for $\eps > 0$) satisfies $|g| > \mu$ on $U \cap \{\eps > 0\}$.  
Once we prove this claim, the result is an immediate consequence of the flow-box theorem.

To prove this, we use the following fact:  If $f$ and $X$ are respectively a smooth function 
and vector field defined in an open set $V \subset \cR^n$ and $\Psi: W \mapsto V$ is a 
diffeomorphism from another open set $W$ into $V$,  then
$$
\big(\mathcal{L}_X(f)\big) \circ \Psi = \mathcal{L}_{\Psi^* X} \big( f\circ \Psi \big).
$$
In other words, the Lie derivative operation commutes with the pull-back operation.  

We apply this to the vector field $X_\eps$ and to blowing-up map $\Phi$ defined in the proof of the 
Theorem~\ref{theorem-smoothableST}. Recall that $\Phi$ is a diffeomorphism outside the 
exceptional divisor $D$, and therefore by the above identity,
$$
g\circ \Phi = \mathcal{L}_{\Phi^* X_\eps}\big( y \circ \Phi \big).
$$
We are going to compute this expression explicitly using the directional charts.  
Recall that, in $E$ chart, we have
$$
\Phi^* X_\eps = \frac{1}{\bar{\eps}} Y\quad \text{Êand }\quad y \circ \Phi = \bar{\eps}\bar{y}
$$
where $Y$ is the vector field in (\ref {vector-fieldY}).  Using the basic properties of the Lie derivative, we get
$$
\begin{array}{rcll}
g\circ \Phi &=& \displaystyle \mathcal{L}_{\frac{1}{\bar{\eps}}Y}\big( \bar{\eps}\bar{y} \big) \\ \vspace{0.2cm}
& = & \displaystyle \frac{1}{\bar{\eps}}\, \mathcal{L}_{Y}\big( \bar{\eps}\bar{y}\big) & \text{(because $\mathcal{L}_{fX}(g) = f\mathcal{L}_{X}(g)$)}\\ \vspace{0.2cm}
& = & \displaystyle\frac{1}{\bar{\eps}}\, \left( \bar{\eps}\mathcal{L}_{Y}\big( \bar{y}\big) + \bar{y}\mathcal{L}_{Y}\big( \bar{\eps}\big) \right) & \text{(by Leibniz's rule)}\\
& = & \displaystyle\mathcal{L}_{Y}\big( \bar{y}\big)& \text{(because $\mathcal{L}_{Y}(\bar{\eps}) = 0$).}\\
\end{array} 
$$
Now, using the expression of $Y$ given in (\ref{vectorfield-Y2}), we obtain 
$$\mathcal{L}_{Y}\big( \bar{y}\big) =  \frac{1}{2}\left(\psi(x,\bar{y})\, \big(a_+ - a_-\big) +  a_+ + a_-\right).$$  
where $a_+$ and $a_-$ are computed by replacing $x,y$ by $x,\bar{\eps}\bar{y}$, respectively.  
By restricting this expression to the divisor $D = \{\bar{\eps} = 0\}$ and using the expression 
(\ref{YonD}), we easily to see that $|g\circ \Phi| > \mu$ for some $\mu > 0$, uniformly in 
sufficiently small neighborhood $U_1$ of $\Phi^{-1}(p)$ in the domain of the $E$-chart.

\begin{figure}[htbp]
\begin{center}
{ 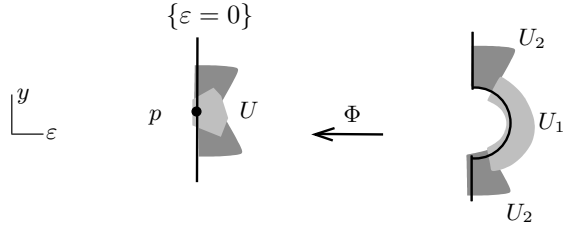}
\caption{The regions $U_1$ and $U_2$.}
\label{flowbox-fig}
\end{center}
\end{figure}

Let us now compute $g\circ \Phi$ in the $F_\pm$ chart.  Analogous computations  gives 
$$
g\circ \Phi = \frac{1}{\tilde{y}} \mathcal{L}_{Z}\big( \tilde{y} \big) .
$$
On the other hand, a simple application of the transformation rules described in the proof of 
Theorem~\ref{theorem-smoothableST} shows that
$$
\mathcal{L}_{Z}(\tilde{y}) = \pm \frac{\tilde{y}}{2}\Big( \psi(x,\frac{1}{\tilde{\eps}})\, \big(a_+ - a_-\big) +  a_+ + a_-\Big)
$$
where $a_+$ and $a_-$ are now computed by replacing $x,y$ by $x,\pm\tilde{y}$, respectively.
Again, by restricting this expression to the divisor $D = \{\tilde{y} = 0\}$ it is easy to see that 
$|g\circ \Phi| > \mu$ uniformly in a sufficiently small neighborhood $U_2$ of $\Phi^{-1}(p)$ in the domain of the $F_\pm$-chart.  

Since the domains of the $E$ and $F_\pm$ charts covers an entire neighborhood of $\Phi^{-1}(p)$ 
in the blowed-up space, it follows that $U_1 \, \cup\,  U_2$ is a neighborhood of $\Phi^{-1}(p)$
 in which  $|g \circ \Phi|Ê> \mu$ .  Hence the inequality $|g|>\mu$ holds in the neighborhood 
 $U = \Phi(U_1 \cup U_2)$ of $p$.  This concludes the proof of Theorem \ref{theorem-sewingreg}.
\end{proof}

We are now ready to conclude the proof of Theorem \ref{theorem-slide}.

\begin{proof}{(end of proof of Theorem \ref{theorem-slide})}
It remains to prove that $\Slide(\vF^\reg) \subset \pi(\vZ^\reg)$.  From the Remark \ref{intersection-slidesew}, 
we know that $\Slide(\vF^\reg) \cap \Sew(\vF^\reg) = \emptyset$.  Combining with the result of the 
above Theorem, we conclude that $\Slide(\vF^\reg) \cap \pi(\vZ^\reg)^\complement = \emptyset$.  
This concludes the proof.
\end{proof}

\begin{Remark}
Recall that in the case of the Sotomayor-Teixeira regularization we require the transition function 
$\psi$ to be strictly monotone in the interval $(-1,1)$.  In this case, the set $\pi(\vZ^\reg)$ can be 
alternatively described by the condition
$$
a_+ \cdot a_- \le 0
$$
In other words, we recover the usual sliding condition of Filippov.  Correspondingly, in this case 
$\pi(\vZ^\reg)^\complement$ is defined by
$$
a_+ \cdot a_- > 0
$$
which corresponds to the Fillipov's sewing condition.
\end{Remark}
\begin{Remark}
Notice that in the limit dynamics defined in the sliding region can be highly dependent on the choice of 
the transition function used in the regularization. To illustrate this, consider the following simple example.  
Let $\vF$ be the discontinuous 1-foliation in $\R^2$ defined by the vector field
$$
X = \big( x^2 + y \big) \frac{\partial}{\partial x} - \sgn(y) \frac{\partial}{\partial y}
$$
with discontinuity locus $\Sigma = \{y = 0\}$.  Given a monotone transition function $\phi$, the 
Sotomayor-Teixeira regularization $\vF^\reg$ is defined by the vector field
$$
X_\eps = \big( x^2 + y \big) \frac{\partial}{\partial x} - \phi\left( \frac{y}{\eps} \right) \frac{\partial}{\partial y}
$$
and it is easy to see that the sliding region coincides with $\Sigma$.  Explicitly, if $t_0 \in (-1,1)$ denotes 
the unique zero of the transition function $\phi$ then the curve $S_\eps = \{y = t_0 \eps\}$ is 
invariant by the flow of $X_\eps$, for each $\eps > 0$.  Notice that the flow of $X_\eps$ {\em restricted to $S_\eps$}
 is defined by the one-dimensional vector field
$$
\big( x^2 + t_0 \eps \big)  \frac{\partial}{\partial x}.
$$ 
In particular, we have three completely distinct topological behaviors depending on the sign of $t_0$.
\end{Remark}

\section{Acknowledgments} The first  author is partially supported by  FAPESP. The second 
author is partially supported by  CAPES, CNPq, 
FAPESP, FP7-PEOPLE-2012-IRSES 318999, and PHB 2009-0025-PC.

\bibliographystyle{plain}

\bibliography{mybibliography}

\end{document}

%% file: example3.pdf_t
\begin{picture}(0,0)%
\includegraphics{example3.pdf}%
\end{picture}%
\setlength{\unitlength}{3947sp}%
\begingroup\makeatletter\ifx\SetFigFont\undefined%
\gdef\SetFigFont#1#2#3#4#5{%
  \reset@font\fontsize{#1}{#2pt}%
  \fontfamily{#3}\fontseries{#4}\fontshape{#5}%
  \selectfont}%
\fi\endgroup%
\begin{picture}(2724,2675)(119,-2188)
\put(1805,304){\makebox(0,0)[lb]{\smash{{\SetFigFont{12}{14.4}{\rmdefault}{\mddefault}{\updefault}{\color[rgb]{0,0,0}$\Sigma$}%
}}}}
\end{picture}%

%% file: example4.pdf_t
\begin{picture}(0,0)%
\includegraphics{example4.pdf}%
\end{picture}%
\setlength{\unitlength}{3947sp}%
\begingroup\makeatletter\ifx\SetFigFont\undefined%
\gdef\SetFigFont#1#2#3#4#5{%
  \reset@font\fontsize{#1}{#2pt}%
  \fontfamily{#3}\fontseries{#4}\fontshape{#5}%
  \selectfont}%
\fi\endgroup%
\begin{picture}(2810,1798)(1069,-2151)
\put(3758,-497){\makebox(0,0)[lb]{\smash{{\SetFigFont{9}{10.8}{\rmdefault}{\mddefault}{\updefault}{\color[rgb]{0,0,0}$\Sigma = \{f = 0\}$}%
}}}}
\end{picture}%

%% file: examplenonsmooth.pdf_t
\begin{picture}(0,0)%
\includegraphics{examplenonsmooth.pdf}%
\end{picture}%
\setlength{\unitlength}{3947sp}%
\begingroup\makeatletter\ifx\SetFigFont\undefined%
\gdef\SetFigFont#1#2#3#4#5{%
  \reset@font\fontsize{#1}{#2pt}%
  \fontfamily{#3}\fontseries{#4}\fontshape{#5}%
  \selectfont}%
\fi\endgroup%
\begin{picture}(3699,849)(929,-143)
\end{picture}%

%% file: multi-gen.pdf_t
\begin{picture}(0,0)%
\includegraphics{multi-gen.pdf}%
\end{picture}%
\setlength{\unitlength}{3947sp}%
\begingroup\makeatletter\ifx\SetFigFont\undefined%
\gdef\SetFigFont#1#2#3#4#5{%
  \reset@font\fontsize{#1}{#2pt}%
  \fontfamily{#3}\fontseries{#4}\fontshape{#5}%
  \selectfont}%
\fi\endgroup%
\begin{picture}(2867,2547)(1975,-3882)
\put(3462,-1508){\makebox(0,0)[lb]{\smash{{\SetFigFont{10}{12.0}{\rmdefault}{\mddefault}{\updefault}{\color[rgb]{0,0,0}$U$}%
}}}}
\put(3214,-2127){\makebox(0,0)[lb]{\smash{{\SetFigFont{10}{12.0}{\rmdefault}{\mddefault}{\updefault}{\color[rgb]{0,0,0}$U_1$}%
}}}}
\put(3895,-2746){\makebox(0,0)[lb]{\smash{{\SetFigFont{10}{12.0}{\rmdefault}{\mddefault}{\updefault}{\color[rgb]{0,0,0}$U_2$}%
}}}}
\put(3895,-3365){\makebox(0,0)[lb]{\smash{{\SetFigFont{10}{12.0}{\rmdefault}{\mddefault}{\updefault}{\color[rgb]{0,0,0}$U_3$}%
}}}}
\put(2967,-3427){\makebox(0,0)[lb]{\smash{{\SetFigFont{10}{12.0}{\rmdefault}{\mddefault}{\updefault}{\color[rgb]{0,0,0}$U_4$}%
}}}}
\put(2472,-2870){\makebox(0,0)[lb]{\smash{{\SetFigFont{10}{12.0}{\rmdefault}{\mddefault}{\updefault}{\color[rgb]{0,0,0}$U_5$}%
}}}}
\put(4452,-2065){\makebox(0,0)[lb]{\smash{{\SetFigFont{10}{12.0}{\rmdefault}{\mddefault}{\updefault}{\color[rgb]{0,0,0}$\Sigma$}%
}}}}
\put(2499,-2322){\makebox(0,0)[lb]{\smash{{\SetFigFont{10}{12.0}{\rmdefault}{\mddefault}{\updefault}{\color[rgb]{0,0,0}$U_6$}%
}}}}
\end{picture}%

%% file: flatcone.pdf_t
\begin{picture}(0,0)%
\includegraphics{flatcone.pdf}%
\end{picture}%
\setlength{\unitlength}{3947sp}%
\begingroup\makeatletter\ifx\SetFigFont\undefined%
\gdef\SetFigFont#1#2#3#4#5{%
  \reset@font\fontsize{#1}{#2pt}%
  \fontfamily{#3}\fontseries{#4}\fontshape{#5}%
  \selectfont}%
\fi\endgroup%
\begin{picture}(2330,3222)(2836,-2671)
\put(4962,-1098){\makebox(0,0)[lb]{\smash{{\SetFigFont{10}{12.0}{\rmdefault}{\mddefault}{\updefault}{\color[rgb]{0,0,0}$\varphi(yz^l/e^{-1/z^2})\in [-1,1]$}%
}}}}
\put(3751,-661){\makebox(0,0)[lb]{\smash{{\SetFigFont{12}{14.4}{\rmdefault}{\mddefault}{\updefault}{\color[rgb]{0,0,0}$\Sigma$}%
}}}}
\put(2851,-2611){\makebox(0,0)[lb]{\smash{{\SetFigFont{10}{12.0}{\rmdefault}{\mddefault}{\updefault}{\color[rgb]{0,0,0}$\{z=0\}$}%
}}}}
\end{picture}%

%% file: blow-upsmooth.pdf_t
\begin{picture}(0,0)%
\includegraphics{blow-upsmooth.pdf}%
\end{picture}%
\setlength{\unitlength}{3947sp}%
\begingroup\makeatletter\ifx\SetFigFont\undefined%
\gdef\SetFigFont#1#2#3#4#5{%
  \reset@font\fontsize{#1}{#2pt}%
  \fontfamily{#3}\fontseries{#4}\fontshape{#5}%
  \selectfont}%
\fi\endgroup%
\begin{picture}(5530,2330)(514,-1682)
\put(5035,-1496){\makebox(0,0)[lb]{\smash{{\SetFigFont{12}{14.4}{\rmdefault}{\mddefault}{\updefault}{\color[rgb]{0,0,0}$\Omega$}%
}}}}
\put(1772,-97){\makebox(0,0)[lb]{\smash{{\SetFigFont{12}{14.4}{\rmdefault}{\mddefault}{\updefault}{\color[rgb]{0,0,0}$\Sigma^{\mathrm{sing}}$}%
}}}}
\put(5232,-206){\makebox(0,0)[lb]{\smash{{\SetFigFont{12}{14.4}{\rmdefault}{\mddefault}{\updefault}{\color[rgb]{0,0,0}$D$}%
}}}}
\put(3406,-369){\makebox(0,0)[lb]{\smash{{\SetFigFont{12}{14.4}{\rmdefault}{\mddefault}{\updefault}{\color[rgb]{0,0,0}$\Phi$}%
}}}}
\put(2027,-1523){\makebox(0,0)[lb]{\smash{{\SetFigFont{12}{14.4}{\rmdefault}{\mddefault}{\updefault}{\color[rgb]{0,0,0}$\Sigma$}%
}}}}
\end{picture}%

%% file: break-saddle.pdf_t
\begin{picture}(0,0)%
\includegraphics{break-saddle.pdf}%
\end{picture}%
\setlength{\unitlength}{3947sp}%
\begingroup\makeatletter\ifx\SetFigFont\undefined%
\gdef\SetFigFont#1#2#3#4#5{%
  \reset@font\fontsize{#1}{#2pt}%
  \fontfamily{#3}\fontseries{#4}\fontshape{#5}%
  \selectfont}%
\fi\endgroup%
\begin{picture}(4169,1196)(2824,-2044)
\put(4731,-1966){\makebox(0,0)[lb]{\smash{{\SetFigFont{12}{14.4}{\rmdefault}{\mddefault}{\updefault}{\color[rgb]{0,0,0}$\Sigma$}%
}}}}
\put(4354,-1031){\makebox(0,0)[lb]{\smash{{\SetFigFont{12}{14.4}{\rmdefault}{\mddefault}{\updefault}{\color[rgb]{0,0,0}$X_0$}%
}}}}
\put(6851,-1061){\makebox(0,0)[lb]{\smash{{\SetFigFont{12}{14.4}{\rmdefault}{\mddefault}{\updefault}{\color[rgb]{0,0,0}$X_\eps$}%
}}}}
\end{picture}%

%% file: reg.pdf_t
\begin{picture}(0,0)%
\includegraphics{reg.pdf}%
\end{picture}%
\setlength{\unitlength}{3947sp}%
\begingroup\makeatletter\ifx\SetFigFont\undefined%
\gdef\SetFigFont#1#2#3#4#5{%
  \reset@font\fontsize{#1}{#2pt}%
  \fontfamily{#3}\fontseries{#4}\fontshape{#5}%
  \selectfont}%
\fi\endgroup%
\begin{picture}(4692,1232)(428,-2204)
\put(3868,-2096){\makebox(0,0)[lb]{\smash{{\SetFigFont{11}{13.2}{\rmdefault}{\mddefault}{\updefault}{\color[rgb]{0,0,0}$\vF_\eps^\reg$}%
}}}}
\put(1055,-2131){\makebox(0,0)[lb]{\smash{{\SetFigFont{11}{13.2}{\rmdefault}{\mddefault}{\updefault}{\color[rgb]{0,0,0}$\vF_0^\reg$}%
}}}}
\end{picture}%

%% file: double-reg2.pdf_t
\begin{picture}(0,0)%
\includegraphics{double-reg2.pdf}%
\end{picture}%
\setlength{\unitlength}{3947sp}%
\begingroup\makeatletter\ifx\SetFigFont\undefined%
\gdef\SetFigFont#1#2#3#4#5{%
  \reset@font\fontsize{#1}{#2pt}%
  \fontfamily{#3}\fontseries{#4}\fontshape{#5}%
  \selectfont}%
\fi\endgroup%
\begin{picture}(4856,4549)(371,-5404)
\put(1803,-3436){\makebox(0,0)[lb]{\smash{{\SetFigFont{9}{10.8}{\rmdefault}{\mddefault}{\updefault}{\color[rgb]{0,0,0}$\vF_{0}$}%
}}}}
\put(4082,-984){\makebox(0,0)[lb]{\smash{{\SetFigFont{9}{10.8}{\rmdefault}{\mddefault}{\updefault}{\color[rgb]{0,0,0}$2\eps$}%
}}}}
\put(5212,-2062){\makebox(0,0)[lb]{\smash{{\SetFigFont{9}{10.8}{\rmdefault}{\mddefault}{\updefault}{\color[rgb]{0,0,0}$2\eta$}%
}}}}
\put(4367,-1366){\makebox(0,0)[lb]{\smash{{\SetFigFont{9}{10.8}{\rmdefault}{\mddefault}{\updefault}{\color[rgb]{0,0,0}$\vF_{\eps,\eta}$}%
}}}}
\put(4007,-3285){\makebox(0,0)[lb]{\smash{{\SetFigFont{8}{9.6}{\rmdefault}{\mddefault}{\updefault}{\color[rgb]{0,0,0}$2\eps$}%
}}}}
\put(4287,-3660){\makebox(0,0)[lb]{\smash{{\SetFigFont{8}{9.6}{\rmdefault}{\mddefault}{\updefault}{\color[rgb]{0,0,0}$\vF_{\eps,0}$}%
}}}}
\put(2588,-1950){\makebox(0,0)[lb]{\smash{{\SetFigFont{9}{10.8}{\rmdefault}{\mddefault}{\updefault}{\color[rgb]{0,0,0}$2\eta$}%
}}}}
\put(1571,-1275){\makebox(0,0)[lb]{\smash{{\SetFigFont{9}{10.8}{\rmdefault}{\mddefault}{\updefault}{\color[rgb]{0,0,0}$\vF_{0,\eta}$}%
}}}}
\put(2262,-2542){\makebox(0,0)[lb]{\smash{{\SetFigFont{8}{9.6}{\rmdefault}{\mddefault}{\updefault}{\color[rgb]{0,0,0}$\eta$}%
}}}}
\put(3061,-3468){\makebox(0,0)[lb]{\smash{{\SetFigFont{8}{9.6}{\rmdefault}{\mddefault}{\updefault}{\color[rgb]{0,0,0}$\eps$}%
}}}}
\end{picture}%

%% file: slide.pdf_t
\begin{picture}(0,0)%
\includegraphics{slide.pdf}%
\end{picture}%
\setlength{\unitlength}{3947sp}%
\begingroup\makeatletter\ifx\SetFigFont\undefined%
\gdef\SetFigFont#1#2#3#4#5{%
  \reset@font\fontsize{#1}{#2pt}%
  \fontfamily{#3}\fontseries{#4}\fontshape{#5}%
  \selectfont}%
\fi\endgroup%
\begin{picture}(1678,2313)(2144,-1687)
\put(2821,-354){\makebox(0,0)[lb]{\smash{{\SetFigFont{10}{12.0}{\rmdefault}{\mddefault}{\updefault}{\color[rgb]{0,0,0}$p$}%
}}}}
\put(3751,332){\makebox(0,0)[lb]{\smash{{\SetFigFont{10}{12.0}{\rmdefault}{\mddefault}{\updefault}{\color[rgb]{0,0,0}$\Sigma$}%
}}}}
\put(3807,-496){\makebox(0,0)[lb]{\smash{{\SetFigFont{10}{12.0}{\rmdefault}{\mddefault}{\updefault}{\color[rgb]{0,0,0}$S_\eps$}%
}}}}
\put(2641,-1614){\makebox(0,0)[lb]{\smash{{\SetFigFont{11}{13.2}{\rmdefault}{\mddefault}{\updefault}{\color[rgb]{0,0,0}$\eps$}%
}}}}
\put(3324,362){\makebox(0,0)[lb]{\smash{{\SetFigFont{10}{12.0}{\rmdefault}{\mddefault}{\updefault}{\color[rgb]{0,0,0}$M$}%
}}}}
\end{picture}%

%% file: slide-fold.pdf_t
\begin{picture}(0,0)%
\includegraphics{slide-fold.pdf}%
\end{picture}%
\setlength{\unitlength}{3947sp}%
\begingroup\makeatletter\ifx\SetFigFont\undefined%
\gdef\SetFigFont#1#2#3#4#5{%
  \reset@font\fontsize{#1}{#2pt}%
  \fontfamily{#3}\fontseries{#4}\fontshape{#5}%
  \selectfont}%
\fi\endgroup%
\begin{picture}(2429,1118)(954,-830)
\put(3368,-398){\makebox(0,0)[lb]{\smash{{\SetFigFont{12}{14.4}{\rmdefault}{\mddefault}{\updefault}{\color[rgb]{0,0,0}$\Sigma$}%
}}}}
\end{picture}%

%% file: transition-f.pdf_t
\begin{picture}(0,0)%
\includegraphics{transition-f.pdf}%
\end{picture}%
\setlength{\unitlength}{3947sp}%
\begingroup\makeatletter\ifx\SetFigFont\undefined%
\gdef\SetFigFont#1#2#3#4#5{%
  \reset@font\fontsize{#1}{#2pt}%
  \fontfamily{#3}\fontseries{#4}\fontshape{#5}%
  \selectfont}%
\fi\endgroup%
\begin{picture}(3452,1971)(1889,-1379)
\put(3355,207){\makebox(0,0)[rb]{\smash{{\SetFigFont{11}{13.2}{\rmdefault}{\mddefault}{\updefault}{\color[rgb]{0,0,0}$2$}%
}}}}
\put(2501,-619){\makebox(0,0)[rb]{\smash{{\SetFigFont{11}{13.2}{\rmdefault}{\mddefault}{\updefault}{\color[rgb]{0,0,0}$-1$}%
}}}}
\put(4561,-936){\makebox(0,0)[rb]{\smash{{\SetFigFont{11}{13.2}{\rmdefault}{\mddefault}{\updefault}{\color[rgb]{0,0,0}$1$}%
}}}}
\put(5326,-811){\makebox(0,0)[rb]{\smash{{\SetFigFont{11}{13.2}{\rmdefault}{\mddefault}{\updefault}{\color[rgb]{0,0,0}$t$}%
}}}}
\put(5251,164){\makebox(0,0)[rb]{\smash{{\SetFigFont{11}{13.2}{\rmdefault}{\mddefault}{\updefault}{\color[rgb]{0,0,0}$\psi(t)$}%
}}}}
\end{picture}%

%% file: strat-slide.pdf_t
\begin{picture}(0,0)%
\includegraphics{strat-slide.pdf}%
\end{picture}%
\setlength{\unitlength}{3947sp}%
\begingroup\makeatletter\ifx\SetFigFont\undefined%
\gdef\SetFigFont#1#2#3#4#5{%
  \reset@font\fontsize{#1}{#2pt}%
  \fontfamily{#3}\fontseries{#4}\fontshape{#5}%
  \selectfont}%
\fi\endgroup%
\begin{picture}(5040,2633)(117,-2009)
\put(5142,-342){\makebox(0,0)[lb]{\smash{{\SetFigFont{11}{13.2}{\rmdefault}{\mddefault}{\updefault}{\color[rgb]{0,0,0}$\{x=\eps t_0\}$}%
}}}}
\put(3443,-1936){\makebox(0,0)[lb]{\smash{{\SetFigFont{11}{13.2}{\rmdefault}{\mddefault}{\updefault}{\color[rgb]{0,0,0}$\eps,\eta > 0$}%
}}}}
\put(751,-1936){\makebox(0,0)[lb]{\smash{{\SetFigFont{11}{13.2}{\rmdefault}{\mddefault}{\updefault}{\color[rgb]{0,0,0}$\eps,\eta = 0$}%
}}}}
\put(1532,410){\makebox(0,0)[lb]{\smash{{\SetFigFont{11}{13.2}{\rmdefault}{\mddefault}{\updefault}{\color[rgb]{0,0,0}$\vF$}%
}}}}
\put(4160,453){\makebox(0,0)[lb]{\smash{{\SetFigFont{11}{13.2}{\rmdefault}{\mddefault}{\updefault}{\color[rgb]{0,0,0}$\{y=\eta u_0\}$}%
}}}}
\end{picture}%

%% file: smooth-ST.pdf_t
\begin{picture}(0,0)%
\includegraphics{smooth-ST.pdf}%
\end{picture}%
\setlength{\unitlength}{3947sp}%
\begingroup\makeatletter\ifx\SetFigFont\undefined%
\gdef\SetFigFont#1#2#3#4#5{%
  \reset@font\fontsize{#1}{#2pt}%
  \fontfamily{#3}\fontseries{#4}\fontshape{#5}%
  \selectfont}%
\fi\endgroup%
\begin{picture}(4960,2196)(1176,-1560)
\put(1461,-439){\makebox(0,0)[lb]{\smash{{\SetFigFont{10}{12.0}{\rmdefault}{\mddefault}{\updefault}{\color[rgb]{0,0,0}$y$}%
}}}}
\put(1191,-984){\makebox(0,0)[lb]{\smash{{\SetFigFont{10}{12.0}{\rmdefault}{\mddefault}{\updefault}{\color[rgb]{0,0,0}$x$}%
}}}}
\put(3836,-610){\makebox(0,0)[lb]{\smash{{\SetFigFont{10}{12.0}{\rmdefault}{\mddefault}{\updefault}{\color[rgb]{0,0,0}$\Phi$}%
}}}}
\put(6121, 81){\makebox(0,0)[lb]{\smash{{\SetFigFont{10}{12.0}{\rmdefault}{\mddefault}{\updefault}{\color[rgb]{0,0,0}$D$}%
}}}}
\put(3262,-72){\makebox(0,0)[lb]{\smash{{\SetFigFont{10}{12.0}{\rmdefault}{\mddefault}{\updefault}{\color[rgb]{0,0,0}$\Sigma$}%
}}}}
\put(1705,-867){\makebox(0,0)[lb]{\smash{{\SetFigFont{11}{13.2}{\rmdefault}{\mddefault}{\updefault}{\color[rgb]{0,0,0}$\eps$}%
}}}}
\end{picture}%

%% file: smooth-ST-NH.pdf_t
\begin{picture}(0,0)%
\includegraphics{smooth-ST-NH.pdf}%
\end{picture}%
\setlength{\unitlength}{3947sp}%
\begingroup\makeatletter\ifx\SetFigFont\undefined%
\gdef\SetFigFont#1#2#3#4#5{%
  \reset@font\fontsize{#1}{#2pt}%
  \fontfamily{#3}\fontseries{#4}\fontshape{#5}%
  \selectfont}%
\fi\endgroup%
\begin{picture}(3971,2210)(2144,-1560)
\put(3836,-610){\makebox(0,0)[lb]{\smash{{\SetFigFont{10}{12.0}{\rmdefault}{\mddefault}{\updefault}{\color[rgb]{0,0,0}$\Phi$}%
}}}}
\put(2765,-404){\makebox(0,0)[lb]{\smash{{\SetFigFont{10}{12.0}{\rmdefault}{\mddefault}{\updefault}{\color[rgb]{0,0,0}$p$}%
}}}}
\put(5741,-1031){\makebox(0,0)[lb]{\smash{{\SetFigFont{12}{14.4}{\rmdefault}{\mddefault}{\updefault}{\color[rgb]{0,0,0}$W$}%
}}}}
\put(2851,-961){\makebox(0,0)[lb]{\smash{{\SetFigFont{12}{14.4}{\rmdefault}{\mddefault}{\updefault}{\color[rgb]{0,0,0}$S$}%
}}}}
\put(5517,-314){\makebox(0,0)[lb]{\smash{{\SetFigFont{10}{12.0}{\rmdefault}{\mddefault}{\updefault}{\color[rgb]{0,0,0}$H_0$}%
}}}}
\put(4880,503){\makebox(0,0)[lb]{\smash{{\SetFigFont{10}{12.0}{\rmdefault}{\mddefault}{\updefault}{\color[rgb]{0,0,0}$\Phi^{-1}(p)$}%
}}}}
\end{picture}%

%% file: sew-reg.pdf_t
\begin{picture}(0,0)%
\includegraphics{sew-reg.pdf}%
\end{picture}%
\setlength{\unitlength}{3947sp}%
\begingroup\makeatletter\ifx\SetFigFont\undefined%
\gdef\SetFigFont#1#2#3#4#5{%
  \reset@font\fontsize{#1}{#2pt}%
  \fontfamily{#3}\fontseries{#4}\fontshape{#5}%
  \selectfont}%
\fi\endgroup%
\begin{picture}(1497,2186)(2144,-1560)
\put(3626,-341){\makebox(0,0)[lb]{\smash{{\SetFigFont{10}{12.0}{\rmdefault}{\mddefault}{\updefault}{\color[rgb]{0,0,0}$U\times(\R^+,0)$}%
}}}}
\end{picture}%

%% file: flow-box.pdf_t
\begin{picture}(0,0)%
\includegraphics{flow-box.pdf}%
\end{picture}%
\setlength{\unitlength}{3947sp}%
\begingroup\makeatletter\ifx\SetFigFont\undefined%
\gdef\SetFigFont#1#2#3#4#5{%
  \reset@font\fontsize{#1}{#2pt}%
  \fontfamily{#3}\fontseries{#4}\fontshape{#5}%
  \selectfont}%
\fi\endgroup%
\begin{picture}(3336,1447)(1516,-565)
\put(3615,114){\makebox(0,0)[lb]{\smash{{\SetFigFont{9}{10.8}{\rmdefault}{\mddefault}{\updefault}{\color[rgb]{0,0,0}$\Phi$}%
}}}}
\put(4714,595){\makebox(0,0)[lb]{\smash{{\SetFigFont{9}{10.8}{\rmdefault}{\mddefault}{\updefault}{\color[rgb]{0,0,0}$U_2$}%
}}}}
\put(4636,-502){\makebox(0,0)[lb]{\smash{{\SetFigFont{9}{10.8}{\rmdefault}{\mddefault}{\updefault}{\color[rgb]{0,0,0}$U_2$}%
}}}}
\put(2501,735){\makebox(0,0)[lb]{\smash{{\SetFigFont{10}{12.0}{\rmdefault}{\mddefault}{\updefault}{\color[rgb]{0,0,0}$\{\eps = 0\}$}%
}}}}
\put(1563,259){\makebox(0,0)[lb]{\smash{{\SetFigFont{9}{10.8}{\rmdefault}{\mddefault}{\updefault}{\color[rgb]{0,0,0}$y$}%
}}}}
\put(1747,  4){\makebox(0,0)[lb]{\smash{{\SetFigFont{9}{10.8}{\rmdefault}{\mddefault}{\updefault}{\color[rgb]{0,0,0}$\eps$}%
}}}}
\put(2962,137){\makebox(0,0)[lb]{\smash{{\SetFigFont{9}{10.8}{\rmdefault}{\mddefault}{\updefault}{\color[rgb]{0,0,0}$U$}%
}}}}
\put(4837, 68){\makebox(0,0)[lb]{\smash{{\SetFigFont{9}{10.8}{\rmdefault}{\mddefault}{\updefault}{\color[rgb]{0,0,0}$U_1$}%
}}}}
\put(2394,127){\makebox(0,0)[lb]{\smash{{\SetFigFont{9}{10.8}{\rmdefault}{\mddefault}{\updefault}{\color[rgb]{0,0,0}$p$}%
}}}}
\end{picture}%